\documentclass{amsart}
\usepackage[margin=3.5cm]{geometry}
\usepackage{xcolor}
\usepackage{amsfonts,amsthm,amssymb,verbatim,amscd,amsmath,blindtext}
\usepackage{multirow}
\usepackage{graphicx}
\usepackage{enumitem}
\usepackage{amssymb}
\usepackage{ytableau}
\usepackage{tikz}
\usepackage{tikz-cd}
\usepackage{float,pifont}
\usepackage{mathtools}
\usepackage[pagebackref,colorlinks=true,citecolor=blue]{hyperref}
\usepackage[capitalise]{cleveref}

\newtheorem{theorem}{Theorem}[section]
\newtheorem{lemma}[theorem]{Lemma}
\newtheorem{corollary}[theorem]{Corollary}
\newtheorem{proposition}[theorem]{Proposition}
\newtheorem{claim}[theorem]{Claim}

\theoremstyle{definition}
\newtheorem{definition}[theorem]{Definition}
\newtheorem{notation}[theorem]{Notation}

\newtheorem{remark}[theorem]{Remark}
\newtheorem{question}[theorem]{Question}

\newtheorem{theoremx}{Theorem}

\newcommand{\bdeg}{\mathbf{deg}}

\DeclareMathOperator{\edim}{edim}
\DeclareMathOperator{\lcm}{lcm}
\DeclareMathOperator{\mingens}{Mingens}
\DeclareMathOperator{\supp}{supp}
\DeclareMathOperator{\pol}{pol}

\DeclareMathOperator{\m}{\mathfrak{m}}
\DeclareMathOperator{\n}{\mathfrak{n}}

\DeclareMathOperator{\Tor}{Tor}

\DeclareMathOperator{\Ht}{ht}

\newcommand{\diamondsymb}[1][1]{%
    \begin{tikzpicture}[scale=#1, %
    baseline=-0.3ex,
    thick,
    ]
    \def\x{2.5mm};
    \def\r{0.11mm};
    \draw (0,0) -- (\x, 0) -- (\x, \x) -- (0, \x) -- cycle;
    \draw (0, 0) -- (\x, \x);

    \draw[fill=black] (0,0) circle (\r);
    \draw[fill=black] (0,\x) circle (\r);
    \draw[fill=black] (\x,0) circle (\r);
    \draw[fill=black] (\x,\x) circle (\r);
    \end{tikzpicture}%
}

\newcommand{\cfoursymbol}[1][1]{%
    \begin{tikzpicture}[scale=#1, %
    baseline=-0.3ex,
    thick,
    ]
    \def\x{2.5mm};
    \def\r{0.11mm};
    \draw (0,0) -- (\x, 0) -- (\x, \x) -- (0, \x) -- cycle;

    \draw[fill=black] (0,0) circle (\r);
    \draw[fill=black] (0,\x) circle (\r);
    \draw[fill=black] (\x,0) circle (\r);
    \draw[fill=black] (\x,\x) circle (\r);
    \end{tikzpicture}%
}
\newcommand{\pfoursymbol}[1][1]{%
    \begin{tikzpicture}[scale=#1, %
    baseline=-0.3ex,
    thick,
    ]
    \def\x{2.5mm};
    \def\r{0.11mm};
    \draw (\x, \x)  -- (\x, 0) -- (0, 0) -- (0, \x);

    \draw[fill=black] (0,0) circle (\r);
    \draw[fill=black] (0,\x) circle (\r);
    \draw[fill=black] (\x,0) circle (\r);
    \draw[fill=black] (\x,\x) circle (\r);
    \end{tikzpicture}%
}

\newcommand{\twoptwosymbol}[1][1]{%
    \begin{tikzpicture}[scale=#1, %
    baseline=-0.3ex,
    thick,
    ]
    \def\x{2.5mm};
    \def\r{0.11mm};
    \draw (0,0) -- (0, \x);
    \draw (\x, 0) -- (\x, \x);

    \draw[fill=black] (0,0) circle (\r);
    \draw[fill=black] (0,\x) circle (\r);
    \draw[fill=black] (\x,0) circle (\r);
    \draw[fill=black] (\x,\x) circle (\r);
    \end{tikzpicture}%
}

\newcommand{\Kfoursymb}[1][1]{%
    \begin{tikzpicture}[scale=#1, %
    baseline=-0.3ex,
    thick,
    ]
    \def\x{2.5mm};
    \def\r{0.2mm};
    \draw (0,0) -- (\x, 0) -- (\x, \x) -- (0, \x) -- cycle;
    \draw (0, 0) -- (\x, \x);
    \draw (\x, 0) -- (0, \x);

    \draw[fill=black] (0,0) circle (\r);
    \draw[fill=black] (0,\x) circle (\r);
    \draw[fill=black] (\x,0) circle (\r);
    \draw[fill=black] (\x,\x) circle (\r);
    \end{tikzpicture}%
}

\newcommand{\pawsymb}[1][1]{%
    \begin{tikzpicture}[scale=#1, %
        thick,
    baseline=-1ex
    ]
    \def\x{2.5mm};
    \def\r{0.11mm};
    \draw (\x/2, \x/2) -- (\x, 0) -- (\x/2, -\x/2) -- cycle;
    \draw (\x,0) -- (3*\x/2,0);

    \draw[fill=black] (\x/2, \x/2) circle (\r);
    \draw[fill=black] (\x, 0) circle (\r);
    \draw[fill=black] (\x/2, -\x/2) circle (\r);
    \draw[fill=black] (3*\x/2,0) circle (\r);
    \end{tikzpicture}%
}

\begin{document}
\title{
Golod ideals in combinatorial commutative algebra
}
\author[B. Briggs]{Benjamin Briggs}
\address{Department of Mathematics, Huxley Building, South Kensington Campus, Imperial College London,
London, SW7 2AZ, U.K.}
\email{b.briggs@imperial.ac.uk}

\author[T. Chau]{Trung Chau}
\address{Chennai Mathematical Institute, Siruseri, Kelambakkam, Tamil Nadu 603103, India}
\email{chauchitrung1996@gmail.com}

\author[A. De Stefani]{Alessandro De Stefani}
\address{Dipartimento di Matematica, Universit\`{a} di Genova, Via Dodecaneso 35, 16146 Genova, Italy}
\email{alessandro.destefani@unige.it}


\subjclass[2020]{}

\begin{abstract}
In this article we study the Golod property of standard graded algebras. We show that determinantal ideals, binomial edge ideals, and permanental ideals are Golod if and only if they have a linear resolution. Next, we give a characterization of when cover ideals define Golod rings, exploiting some considerations on multidegrees of Koszul cycles and Massey products. Finally, we show that squarefree strongly Golod ideals (and, more generally, lcm-strongly Golod ideals) are Golod, and not just weakly Golod.
\end{abstract}

\maketitle

\section{Introduction}
Let $(R,\m,\Bbbk)$ denote either a local ring, or a standard graded $\Bbbk$-algebra. If $K^R$ denotes the Koszul complex on a minimal generating set of $\m$, and $P^R_\Bbbk(t) = \sum_{i \geq 0} \dim_\Bbbk(\Tor_i^R(\Bbbk,\Bbbk))t^i$ is the Poincar\'e series of $\Bbbk$, then Serre proved a coefficient-wise inequality
\[
P^R_\Bbbk(t) \leq \frac{(1+t)^{\edim(R)}}{1-\sum_{j\geq 1}\dim_\Bbbk(H_j(K^R))t^{j+1}}
\]
where $\edim(R) = \dim_\Bbbk(\m/\m^2)$ denotes the embedding dimension of $R$. Golod gave a characterization of when equality holds in terms of Massey operations \cite{GolodMassey}, and for this reason a ring whose Poincar\'e series achieves the upper bound is nowadays named after him. Golod rings and Golod homomorphisms have been extensively studied in the literature; we refer the interested reader to \cite{Avramov1998, DDS2022, DESTEFANI2016, Kat17, Levin1975, RossiSega2014,   VdB2022} for some background and some more recent developments. A well-known fact is that if $R$ is Golod, then the multiplication on the Koszul homology $H_{\geq 1}(K^R)$ must be trivial; a ring satisfying this condition was called weakly Golod in \cite{DESTEFANI2016}. Conversely, if one can find a set $\mathcal S$ of Koszul cycles whose classes in $H_{\geq 1}(K^R)$ form a $\Bbbk$-basis, and such that the product of any two elements in $\mathcal S$ is zero -- not just a boundary -- then $R$ is Golod (e.g., see \cite{HH13}).  None of these implications can be reversed in general, not even for quotients defined by monomial ideals (see \cite{Kat17}). For the rest of this introduction, assume that $R=Q/I$, where $(Q,\n,\Bbbk)$ is either a regular local ring or a standard graded polynomial ring over $\Bbbk$. In the latter case, $I$ is also assumed to be homogeneous. We will also assume that $I \subseteq \n^2$. If $\Ht(I) \leq 1$, then $R$ is Golod. Moreover, if $R$ is Gorenstein, then it is Golod only in this case, i.e., only if $I$ is~principal. 

In the graded setup, we recall that $I$ is said to have a linear resolution if it is generated in a single degree $D$, and for any $i$ we have $\Tor^Q_i(\Bbbk,I)_j = 0$ for all $j \ne i+D$. In \cite{BF85}, Backelin and Fr{\"o}berg showed that if $I$ has a linear resolution, then $R$ is Golod. This was later generalized by Herzog, Reiner and Welker, who showed that componentwise linear ideals define Golod quotients \cite{HRW99}. Our first main result shows that having a linear resolution is also a necessary condition for several classes of well-studied rings. We remark that the equivalence between being Golod and having a linear resolution was previously known for some classes of rings, e.g., Koszul rings (\cite[Remark~3.1]{CINR15}) or $2\times 2$ sub-determinantal rings \cite{Omkar}. We remark that part (ii) of the theorem below is a special case of \cite[Theorem~1.1]{Omkar}, which was obtained independently using a different~technique.  
\begin{theoremx} \label{THMX linear res}
 Let $I$ be a homogeneous ideal in a polynomial ring over a field given by one of the following:
    \begin{enumerate}[label=(\roman*)]
        \item the determinantal ideal of a generic, generic symmetric, or generic Hankel matrix;
        \item the binomial edge ideal of a finite simple graph;
        \item the $2\times 2$ permanental ideal of a generic matrix.
    \end{enumerate}
    Then $R=Q/I$ is Golod if and only if $I$ has a linear resolution.
\end{theoremx}

Another well studied  class of ideals that we will give a full characterization for Golodness is cover ideals of graphs (we refer to Section~\ref{sec:multideg-tech} for the exact definition). Cover ideals emerged naturally as the Alexander dual of the better known edge ideals of graphs. This led to a systematic investigation of their homological invariants, ordinary powers, symbolic powers, and asymptotic behavior \cite{CPSTY, ER98,FHvT11,FVT-sequentially-CM, HvT22,Betti-cover-ideals}. Interestingly, cover ideals have been established to be counterexamples to conjectures due to their erratic behaviors \cite{HS15-cover-ideals,cover-ideals-counter}. Somewhat surprisingly, despite the large body of literature on cover ideals, a complete characterization of  componentwise linear cover ideals remains an interesting open problem (we refer to \cite{FVT-sequentially-CM,Woodroofe-vertex-decomposable} for partial answers). What we prove is that, most of the time, cover ideals define Golod rings, despite not always being componentwise linear. This marks an important difference with Theorem~\ref{THMX linear res}, and contributes to the literature yet another ground where the Golod property is ubiquitous.

\begin{theoremx} \label{THMX graph}
    For the cover ideal $J(G)$ of a finite simple graph $G$, the following conditions are equivalent:
    \begin{enumerate}[label=(\roman*)]
        \item $J(G)$ is not Golod;
        \item $J(G)$ is a complete intersection of two monomials of degree at least two;
        \item $J=K_{m,n}$ is a complete bipartite graph with $\min\{m,n\} \geq 2$.
    \end{enumerate}
\end{theoremx}

With the exception of special classes such as the ones listed above, establishing whether a ring is Golod or not can be a very challenging task. In \cite{HH13}, Herzog and Huneke introduced a notion for a homogeneous ideal $I$ in a standard graded polynomial ring $Q=\Bbbk[x_1,\ldots,x_n]$ over a field $\Bbbk$ of characteristic zero: $I$ is called strongly Golod if $\partial_{x_i}(f)\partial_{x_j}(g) \in I$ for all $f,g \in I$. Exploiting an explicit technique to lift homology classes of $\Tor_\bullet^Q(Q/I,\Bbbk)$ to $K^{Q/I}$, described in \cite{Herzog91}, they showed that strongly Golod ideals have trivial product on the Koszul complex, and hence they define Golod rings. A big advantage of considering the strongly Golod notion is that it is a much easier condition to test than the one of being Golod. For instance, it allows to deduce that, in the context of its definition, powers, symbolic powers and certain saturations and closures of ideals are always Golod (see \cite[Theorem 2.3, Theorem 2.11 and Proposition 2.13]{HH13}). The notion of strongly Golod ideals was later revisited by Herzog and Maleki to treat the case when $\Bbbk$ has positive characteristic \cite{HM18}; however, for the sake of exposition, in this introduction we will assume that $\Bbbk$ has characteristic zero. Note that being strongly Golod is far from being a necessary condition for Golodness. For instance, squarefree monomial ideals are essentially never strongly Golod. For this reason, Herzog and Huneke in the same paper defined the notion of squarefree strongly Golod ideal, which applies to squarefree monomial ideals. The difference with the original definition is that, if $f,g$ are monimials in $I$, one has to test whether $\partial_{x_i}(f)\partial_{x_j}(g) \in I$ only if such a product is squarefree. In \cite[Theorem 3.5]{HH13}, they claimed that squarefree strongly Golod ideals define Golod rings; however, their proof relied on an incorrect result of Berglund and J{\"o}llenbeck (see \cite[Theorem 3]{BJ07} and \cite{Kat17}). What can be deduced from their proof, however, is that squarefree strongly Golod monomial ideals define weakly Golod quotients. The third author introduced the class of lcm-strongly Golod monomial ideals as a generalization of being squarefree strongly Golod, and proved that lcm-strongly Golod ideals also define weakly Golod quotients (see \cite[Theorem 4.6]{DESTEFANI2016}). In this paper, we slightly modify the notion of lcm-strongly Golod ideal taking into account the Discussion following \cite[Theorem 4.6]{DESTEFANI2016}, to allow for a larger class of ideals. Given a monomial ideal $I$, with minimal generating set $f_1,\ldots,f_t$, we let $\lcm(I) = \lcm(f_1,\ldots,f_t)$. We say that $I$ is \emph{lcm-strongly Golod} if $\partial_{x_i}(f)\partial_{x_j}(g) \in I$ for any $i \ne j$ and any pair of monomials $f,g \in I$ such that $\partial_{x_i}(f)\partial_{x_j}(g)$ divides $\lcm(I)$.

Our main result in this direction is to show that lcm-strongly Golod ideals are actually Golod.

\begin{theoremx} \label{THMX lcm}
Let $k$ be a field, and $Q=\Bbbk[x_1,\ldots,x_n]$, with the standard grading. Let $I \subseteq Q$ be a monomial ideal. If $I$ is lcm-strongly Golod, then $Q/I$ is Golod.
\end{theoremx}

The key idea in the proof of Theorem \ref{THMX lcm} is to use a modified version $K^{Q/I}_{{\rm lcm}}$ of the Koszul complex $K^{Q/I}$, which is quasi-isomorphic to it as a dg-algebra (see Lemma \ref{lem_Koslcm}). The assumption that $I$ is lcm-strongly Golod guarantees that the multiplication on a suitable set of cyles of $(K^{Q/I}_{\rm lcm})$ is trivial, and as in \cite{HH13} this allows to conclude that $Q/I$ is Golod.

\subsection*{Acknowledgments} Chau is supported by the Infosys Foundation. De Stefani was partially supported by the MIUR Excellence Department Project CUP D33C23001110001, PRIN 2022 Project 2022K48YYP, and by INdAM-GNSAGA. The third named author thanks Aldo Conca for fruitful conversations on the topics of this paper. 






\section{First technique: Restriction and algebra retracts}

Throughout this section, let $\Bbbk$ be an arbitrary field. Given two $\Bbbk$-algebras $R\subseteq S$ we say that $R$ is a retract of $S$ if there exists a $\Bbbk$-algebra homomorphism $\varphi:S \to R$ such that $\varphi(r) = r$ for all $r \in R$. When $R$ and $S$ are non-negatively graded with $R_0=S_0=\Bbbk$, we further tacitly assume that $\varphi$ is homogeneous, necessarily of degree zero.

\begin{lemma}[\protect{\cite[Corollary 2.6]{OHH00}}]\label{lem:Golod-retract}
    Let $R\subseteq S$ be an algebra retract of graded $\Bbbk$-algebras. If $S$ is Golod, then so is $R$.
\end{lemma}


Algebra retracts are well-studied objects in the literature (e.g.\ \cite{EN14,GUPTA2017,OHH00}). We will now describe a well-known way to build algebra retracts in  
a polynomial ring  $Q=\Bbbk[x_1,\dots, x_n]$.

For a set of variables $S=\{x_{i_1},\dots, x_{i_k}\}$ we write $Q_{|S}\coloneqq \Bbbk[x_{i_1},\dots, x_{i_k}]$. Let $I$ be a homogeneous ideal of $Q$. Setting $x_j=0$ for any $x_j\notin S$ yields an algebra map $Q\to Q_{|S}$ and therefore an ideal $IQ_{|S}$. Likewise there is an inclusion $Q_{|S}\subseteq Q$ yielding an ideal $I\cap Q_{|S}$. When these ideals are equal we write 
\[
I_{|S} \coloneqq IQ_{|S}=I\cap  Q_{|S},
\]
and we say that $I_{|S}$ is a \emph{restriction} of $I$ (with respect to $S$). Said in words: an ideal is a restriction of $I$ if it is both generated by elements of $I$ and it is obtained from $I$ by setting some variables equal to zero.

\begin{lemma}\label{lem:restriction-retract}
    If $I_{|S}$ is a restriction of $I$, then $Q_{|S}/I_{|S}$ is an algebra retract of $Q/I$. In particular, if $Q/I$ is Golod then so is $Q_{|S}/I_{|S}$.
\end{lemma}

\begin{proof}
    This follows from the observation that the composition
    \[
    Q_{|S}/I\cap Q_{|S} \to Q/I \to Q_{|S}/IQ_{|S}
    \]
    is the identity map if $I\cap Q_{|S}=IQ_{|S}$. 
    The last statement is due to Lemma~\ref{lem:Golod-retract}.
\end{proof}

\subsection{Determinantal ideals}
As an application of the restriction process, we can describe all Golod determinantal ideals of generic, or generic symmetric, or generic Hankel matrices. We recall some terminology. Let $\Bbbk$ be a field; an $m\times n$ matrix $X$ is called

\begin{itemize}
    \item \emph{generic} if its entries are algebraically independent indeterminates over $\Bbbk$;
    \item \emph{generic symmetric} if $X$ is symmetric (in particular, this means $m=n$) and the entries $\{x_{ij}\}_{1\leq i\leq j\leq n}$ are algebraically independent indeterminates over $\Bbbk$;
    \item \emph{generic Hankel} if $x_{ij}=x_{i+l,j-l}$ for any $i,j,l$, and the entries $\{x_{ij}\}_{i=1 \text{ or } j =n}$ are  algebraically independent indeterminates over $\Bbbk$.
\end{itemize}

A \emph{determinantal ideal} of $X$ is an ideal of $\Bbbk[X]$ that is generated by all $t\times t$ minors of $X$ for a given $t$. The following is well-known among experts, but we will provide a proof here for completeness.

\begin{lemma}\label{lem:determinantal-retracts}
    Let $\Bbbk$ be a field, and $X$ be a generic matrix (respectively, generic symmetric matrix). If $Y$ is a submatrix (respectively, symmetric submatrix) of $X$, then for any positive integer $t$, the natural map $\Bbbk[Y]/I_t(Y)\to \Bbbk[X]/I_t(X)$ is an algebra retract.
\end{lemma}

\begin{proof}
    First assume that $X$ and $Y$ are generic. It suffices to prove the claim when $Y$ is $X$ without its last column. Then for any fixed positive integer $t$, any $t\times t$ minor of $X$, after specializing the variables in the last column of $X$ to $0$, becomes either $0$ (if the submatrix corresponding to the minor has $0$ as the last column) or a $t\times t$ minor of $Y$. In other words, $I_t(Y)$ is a restriction of $I_t(X)$, so the result the follows from Lemma~\ref{lem:restriction-retract}. The case where $X$ and $Y$ are symmetric follows by a similar argument, in which $Y$ can be assumed to be $X$ without its last row and its last column.
\end{proof}

We can now give a characterization of which determinantal ideals are Golod.

\begin{theorem}\label{thm:Golod-determinantal}
    Let $X$ be an $m\times n$  generic, or generic symmetric (in which case $m=n$), or generic Hankel matrix, and let $t$ be an integer. Then the following holds.
    \begin{enumerate}
        \item If $X$ is generic, then $\Bbbk[X]/I_t(X)$ is Golod if and only if $t=\min\{m,n\}$ or $t=1$.
        \item If $X$ is generic symmetric (in which case $m=n$), then $\Bbbk[X]/I_t(X)$ is Golod if and only if $t\in \{1,n-1,n\}$.
        \item If $X$ is generic Hankel, then $\Bbbk[X]/I_t(X)$ is Golod. 
    \end{enumerate}
\end{theorem}

\begin{proof}
    If $t=1$, then $\Bbbk[X]/I_t(X)=\Bbbk$ is a Golod ring. For the rest of the proof, we assume that $t\geq 2$, and $m\leq n$.
    \begin{enumerate}
        \item Assume that $X$ is generic. If $t=m=\min\{m,n\}$, then $\Bbbk[X]/I_t(X)$ has a linear resolution, which can be seen from its minimal Eagon-Northcott resolution (see also \cite[Theorem 8.7.13]{BCRV}). Thus $\Bbbk[X]/I_t(X)$ is Golod by \cite[Theorem 4]{HRW99}, as desired. Now assume that $t<m$. Let $Y$ be the submatrix of $X$ formed by its first $m$ columns.  Since $Y$ is square, $\Bbbk[Y]/I_t(Y)$ is Gorenstein, and since $2\leq t<m$, it is not a hypersurface. Therefore, $\Bbbk[Y]/I_t(Y)$ is not Golod, which implies that $\Bbbk[X]/I_t(X)$ is not Golod by Lemmas~\ref{lem:restriction-retract} and \ref{lem:determinantal-retracts}.
        \item Assume that $X$ is generic symmetric (in which case $m=n$). If $t\in \{n-1,n\}$, then $\Bbbk[X]/I_t(X)$ has a linear resolution (\cite[Theorem~1]{DK24}), and thus is Golod by \cite[Theorem 4]{HRW99}. Now assume that $2\leq t<n-1$. If $n-t$ is even, then $\Bbbk[X]/I_t(X)$ is Gorenstein, and since it is not a hypersurface, it is not Golod either, as desired. Now assume that $n-t$ is odd. Let $Y$  be the square submatrix of $X$ formed by its first $n-1$ rows and $n-1$ columns. Then $\Bbbk[Y]/I_t(Y)$ is Gorenstein, and thus is not Golod since it is not a hypersurface. This implies that $\Bbbk[X]/I_t(X)$ is not Golod by Lemmas~\ref{lem:restriction-retract} and \ref{lem:determinantal-retracts}, as desired.
        \item Assume that $X$ is generic Hankel. Then by \cite[Lemma~2.3]{GP81}, $\Bbbk[X]/I_t(X)$ is isomorphic to $\Bbbk[Y]/I_t(Y)$ where $Y$ is a generic Hankel matrix of size $t\times l$ for some integer $l$. In other words, $\Bbbk[Y]/I_t(Y)$ is defined by maximal minors, and has a linear resolution, which can be seen from its minimal Eagon-Northcott resolution (see also \cite[Theorem 8.7.13]{BCRV}). Therefore, $\Bbbk[X]/I_t(X)$ is Golod by \cite[Theorem 4]{HRW99}, as~desired.\qedhere
    \end{enumerate}
\end{proof}

Looking at the above proof more closely, we see that in all the cases when $\Bbbk[X]/I_t(X)$ is Golod, it in fact has a linear resolution over $\Bbbk[X]$. We therefore have the following corollary.

\begin{corollary}\label{cor:determinantal-Golod-linear-resolution}
    Let  $X$ be an $m\times n$  generic, or generic symmetric, or generic Hankel matrix, and let $t$ be an integer. Then $\Bbbk[X]/I_t(X)$ is Golod if and only if $I_t(X)$ has a linear resolution.
\end{corollary}

\subsection{Binomial edge ideals} For each integer $n\geq 1$, let $[n]$ denote the set $\{1,2,\dots, n\}$. Let $G$ be a finite simple graph with the vertex set $V(G)=[n]$. The \emph{binomial edge ideal} of $G$, which in this article will denoted $J_b(G)$, is defined to be the ideal of $\Bbbk[x_i,y_i\colon i\in V(G)]$ generated by all binomials $x_iy_j-x_jy_i$ where $\{i,j\}$ ranges over the edge set of $G$. 

A subgraph $H$ of a graph $G$ is called \emph{induced} if for any two vertices $i,j$ of $H$ such that $\{i,j\}$ is an edge of $G$, the edge $\{i,j\}$ is also an edge of $H$. Using a similar argument to the proof of Lemma~\ref{lem:determinantal-retracts}, induced subgraphs yield algebra retracts at the level of binomial edge ideals. Thus we state this fact below without proof.

\begin{lemma}\label{lem:binomial-retracts}
    Let $G$ be a finite simple graph with the vertex set $[n]$ and let $H$ be an induced subgraph of $G$. The natural inclusion
    \[
    \Bbbk\big[x_i,y_i\colon i\in V(H)\big]/J_b(H)\longrightarrow\Bbbk\big[x_i,y_i\colon i\in V(G)\big]/J_b(G)
    \]
    is an algebra retract.\qedhere
\end{lemma}

We can now classify which binomial edge ideals are Golod. 
Recall that the \emph{complete graph} on $n$ vertices, denoted by $K_n$, is the graph where any two distinct vertices are adjacent.

\begin{theorem}\label{thm:Golod-binomial}
    Let $G$ be a connected finite simple graph with the vertex set $[n]$, where $n$ is an integer, and let $Q=\Bbbk\big[x_i,y_i\colon i\in [n]\big]$. Assume that $G$ has no isolated vertex. Then the following are equivalent:
    \begin{enumerate}
        \item $Q/J_b(G)$ is Golod;
        \item $J_b(G)$ has a linear resolution;
        \item $G$ is a complete graph.
    \end{enumerate}
\end{theorem}
\begin{proof}

    \emph{(1) $\implies $ (3):} 
    Write $P_3$ for the graph with three vertices and two edges. By \cite[Corollary 1.2]{EHH11}, the ring $\Bbbk\big[x_i,y_i\colon i\in [3]\big]/J_b(P_3)$ is a non-hypersurface complete intersection ring, and therefore is not Golod. Thus, if $Q/J_b(G)$ is Golod, then $G$ does not contain $P_3$ as an induced subgraph by Lemma~\ref{lem:binomial-retracts}. Since $G$ is connected, this makes it a complete graph.

    \emph{(3) $\implies $ (2):} if $G$ is a complete graph, then $Q/J_b(G)$ is a determinantal ideal of maximal minors of a generic $2\times n$ matrix, and thus has a linear resolution. 
    
    \emph{(2) $\implies $ (1):} This follows from \cite[Theorem 4]{HRW99}.
\end{proof}


Following similar arguments, it can be shown that $J_b(G)$, where $G$ is a complete graph with possible isolated vertices, are also the only Golod ideals in a larger class called \emph{generalized binomial edge ideals}; see \cite{Ra13} for a definition of generalized binomial edge ideals. We leave the details to the interested reader.

\subsection{Permanental ideals}

As another application of the techniques of this section, we study the class of the permanental rings. Recall that for an $n\times n$ matrix $X$, the \emph{permanent} of $X$ is defined to be
\[
\operatorname{perm} X \coloneqq \sum_{\sigma\in S_n} \prod_{i=1}^n x_{i,\sigma(i)}. 
\]
For an integer $t$ and a generic matrix $X$, let $P_t(X)$ denote the ideal of $\Bbbk[X]$ generated by all $t\times t$ subpermanents of $X$. Note that $P_t(X)=I_t(X)$ when $\Bbbk$ is of characteristic $2$. Similar to determinantal rings, there are natural algebra retracts concerning permanental rings. We hence state the analog of \cref{lem:determinantal-retracts} without proofs.

\begin{lemma}\label{lem:permanental-retracts}
    If $X,Y$ are generic matrices and $Y$ is a submatrix of $X$, then for any positive integer $t$, the natural map $\Bbbk[Y]/P_t(Y)\to \Bbbk[X]/P_t(X)$ is an algebra retract.
\end{lemma}

Under certain conditions, we can obtain a characterization of Golod permanental rings.

\begin{theorem}\label{thm:Golod-permanental}
    Assume that one of the following holds:
    \begin{enumerate}[label=(\roman*)]
        \item $t=2$ and $\operatorname{char} \Bbbk \neq 2$.
        \item $t\in \{3,4\}$ and $\operatorname{char} \Bbbk=0$.
    \end{enumerate}
    Let $X$ be an $m\times n$ generic matrix where $m\leq n$. Then the following are equivalent:
    \begin{enumerate}
        \item $\Bbbk[X]/P_t(X)$ is Golod;
        \item $\Bbbk[X]/P_t(X)$ is a hypersurface;
        \item $P_t(X)$ has a linear resolution;
        \item $(m,n)=(t,t)$.
    \end{enumerate}
\end{theorem}

\begin{proof}
    \emph{(1)$\implies$ (4):} We will use contraposition. Assume that $(m,n)\neq (t,t)$. In particular, we can let $Y$ be a $t\times (t+1)$ submatrix of $X$.  Then $P_t(Y)$ is a non-hypersurface complete intersection (due to \cite[Theorem~4.1]{LS00} if (i) holds and \cite[Proposition~3.10]{BCMV25} if (ii) holds). In particular $\Bbbk[Y]/P_t(Y)$ is not Golod, and thus neither is $\Bbbk[X]/P_t(X)$ by \cref{lem:permanental-retracts}, as desired. 

    \emph{(4) $\implies$ (2)$\implies$ (3):} If $(m,n)=(t,t)$, then $P_t(X) = \left(\operatorname{perm} X\right)$. Thus $\Bbbk[X]/P_t(X)$ is a hypersurface, and has a linear resolution. 

    \emph{(3) $\implies $ (1):} This follows from \cite[Theorem 4]{HRW99}.
\end{proof}

The above proof hinges on \cref{lem:permanental-retracts} and the fact that under certain conditions, $\Bbbk[Y]/P_t(Y)$ is not Golod for a $t\times (t+1)$ generic matrix $Y$. Regarding condition (ii), the fact that $\Bbbk[Y]/P_t(Y)$ is a complete intersection for a $t \times (t+1)$ generic matrix $Y$ is, in fact, conjectured to hold in greater generality (cf. \cite[Conjecture~3.4]{BCMV25}). Therefore a weaker question is whether \cref{thm:Golod-permanental} holds under the same hypotheses as \cite[Conjecture~3.4]{BCMV25}. As a byproduct of Theorem \ref{thm:Golod-permanental}, we obtain a new class of graded rings whose Golodness depends on characteristic.

\begin{corollary}
    Let $X$ be a $2\times n$ generic matrix where $n\geq 3$. Then $\Bbbk[X]/P_2(X)$ is Golod if and only if $\operatorname{char} \Bbbk =2$.
\end{corollary}

\begin{proof}
    This follows from \cref{thm:Golod-determinantal} and \cref{thm:Golod-permanental}.
\end{proof}

We remark that one can ask when $P_t(X)$ is Golod in the cases where $X$ is a generic symmetric/generic Hankel matrix. The situation is similar to generic matrices, in the sense that most known results are about $P_2(X)$, and not much about $P_t(X)$ for $t\geq 3$. Experiments on Macaulay2 raises the following question.

\begin{question}
    Is it true that if $X$ is a generic symmetric or a generic Hankel matrix, then $P_2(X)$ is Golod if and only if $X$ is of size $2\times 2$?
\end{question}

\subsection{Polarization and restriction}

We next focus on monomial ideals. If $I$ is a monomial ideal, it has a unique set of minimal monomial generators, denoted by $\mingens(I)$. If $I$ is a monomial ideal, and $m$ is a monomial, then we denote by $I^{\leq m}$ the monomial ideal generated by monomials in $\mingens(I)$ that divide $m$. This was first introduced by Herzog, Hibi, and Zheng \cite{HHZ04Dirac}, although the idea can be found in older literature (see, e.g., \cite{BPS98}). In general, $I^{\leq m}$ is not a restriction of $I$. 
However, if $I$ is assumed to be squarefree we have the following.

\begin{lemma}\label{lem:squarefree-retract}
    Let $I$ be a squarefree monomial ideal in $Q=\Bbbk[x_1,\dots, x_n]$, and $m$ a monomial. Then $I^{\leq m}$ is a restriction of $I$ (with respect to $\supp(m)$). In particular, if $Q/I$ is Golod, then so is $Q/I^{\leq m}$.
\end{lemma}

\begin{proof}
    Set $S=\supp(m)$ and note that $m=\prod_{x\in S}x$. Since $I$ is squarefree any monomial in $\mingens(I)$ either contains a variable not in $S$, or it divides $m$. Therefore $IQ_{|S}=I^{\leq m}=I\cap Q_{|S}$, so $I_{|S}=I^{\leq m}$ is a restriction of $I$, as desired. The last statement follows from Lemma~\ref{lem:restriction-retract}.
\end{proof}

As mentioned earlier, $I^{\leq m}$ is not necessarily a restriction of $I$ if $I$ is not squarefree. Nevertheless, we will show that Golodness descends from $I$ to $I^{\leq m}$ for any monomial ideal $I$. To do so, we recall the concept of polarization.  

\begin{notation}\label{not:polarization}
    Let $I$ be a monomial ideal in $Q=\Bbbk[x_1,\dots ,x_n]$. For each $i\in [n]$, let $e_i$ denote the highest power of $x_i$ among the minimal generators of $I$. For each $m=x_1^{a_n}\cdots x_n^{a_n}\in \mingens(I)$, the \emph{polarization} of $m$ is the monomial 
        \[
    \widetilde{m}= x_{1,1}x_{1,2}\cdots x_{1,a_1}x_{2,1}x_{2,2}\cdots x_{2,a_2} \cdots x_{n,1}x_{n,2}\cdots x_{n,a_n}
    \]
    in the polynomial ring $Q_{\pol(I)}=\Bbbk[x_{i,j}\colon i\in [n],j\in [e_i]]$. If $\mingens(I)=\{m_1,\dots, m_q\}$, then the \emph{polarization} of $I$, denoted by $\pol(I)$, is the monomial ideal (minimally) generated by $\widetilde{m_1},\dots, \widetilde{m_q}$ in the ring $Q_{\pol}$. 
\end{notation}

\begin{lemma}\label{lem:pol-vs-restriction}
    Use the set-up as in Notation~\ref{not:polarization}. Let $m$ be a monomial in $Q$. Then the following holds:
    \begin{enumerate}
        \item $Q/I$ is Golod if and only if $Q_{\pol} /\pol(I)$ is Golod;
        \item $\pol(I^{\leq m})Q_{\pol(I)}= \pol(I)^{\leq \widetilde{m}}$.
    \end{enumerate}
\end{lemma}

\begin{proof}
    \begin{enumerate}
        \item It is known from \cite[Proposition~2.1]{Swan06} that $Q/I$ is isomorphic to $Q_{\pol} /\pol(I)$ modulo a linear regular sequence. The result then follows from \cite[Proposition~5.2.4]{Avramov1998}.
        \item It is straightforward that for any monomial $m'\in \mingens(I)$, $m'$ divides $m$ if and only if $\widetilde{m'}$ divides $\widetilde{m}$. We then have 
        \begin{align*}
            \pol(I^{\leq m})Q_{\pol(I)} &= (\widetilde{m'}\colon m'\in \mingens(I), m'\mid m)\\
            &= (\widetilde{m'}\colon \widetilde{m'}\in \mingens(\pol(I)), \widetilde{m'}\mid \widetilde{m})\\
            &= (m''\colon m''\in \mingens(\pol(I)), m''\mid \widetilde{m})\\
            &= \pol(I)^{\leq \widetilde{m}},
        \end{align*}
        as desired.\qedhere
    \end{enumerate}
\end{proof}

\begin{theorem}\label{thm:Golod-restriction}
    Let $I$ be a monomial ideal in a polynomial ring $Q$ and $m$ a monomial. If $Q/I$ is Golod, then so is $Q/I^{\leq m}$.
\end{theorem}

\begin{proof}
    We will use the notation as in Notation~\ref{not:polarization}. Since $Q/I$ is Golod, so is $Q_{\pol(I)}/\pol(I)$ by Lemma~\ref{lem:pol-vs-restriction} (1), and thus so is $Q_{\pol(I)}/ \pol(I^{\leq m})Q_{\pol(I)}$ by Lemmas~\ref{lem:restriction-retract} and \ref{lem:pol-vs-restriction} (2). Set $m=x_1^{a_1}\cdots x_n^{a_n}$. We have that \[
    Q_{\pol(I)}/ \pol(I^{\leq m})Q_{\pol(I)} = Q_{\pol(I^{\leq m})}/\pol(I^{\leq m})[x_{i,j}\colon i\in [n], j\in [a_i+1,e_i]]
    \]
    and thus $Q_{\pol(I^{\leq m})}/\pol(I^{\leq m})$ is Golod by \cite[Proposition~5.2.4]{Avramov1998}. Therefore, $I^{\leq m}$ is Golod by Lemma~\ref{lem:pol-vs-restriction} (1), as desired.
\end{proof}

\section{Second technique: Multidegrees and lcm lattices}\label{sec:multideg-tech}

Let $\Bbbk$ be a field, $Q=\Bbbk[x_1,\ldots,x_n]$ with the standard grading, and $\m = (x_1,\ldots,x_n)Q$. Let $I \subseteq Q$ be a homogeneous ideal, and set $R=Q/I$. We now briefly recall the definition of Massey products, and point the interested reader to \cite{GolodMassey} (see also \cite{Avramov1998}, or \cite{Kat17}) for more information.


For $t=2$, given $h_1,h_2 \in H_{\geq 1}(K^R)$ the binary Massey product $\mu_2(h_1,h_2)$ is the usual product $h_1h_2$ on $H_*(K^R)$. For $t \geq 3$, $\mu_t$ is a partially-defined function that assigns to $h_1,\ldots,h_t \in H_{\geq 1}(K^R)$ a set $\mu_t(h_1,\ldots,h_t) \subseteq H_*(K^R)$, and it is defined if and only if for all $1 \leq i \leq j \leq t$ with $(i,j) \ne (1,t)$ there exist $h_{i,j} \in K^R$ such that
\begin{enumerate}
    \item $h_{i,i} \in Z(K^R)$ and $[h_{i,i}]= h_i$ for all $i=1,\ldots,t$;
    \item $\partial(h_{i,j}) = \sum_{k=i}^j \overline{h_{i,k}}h_{k+1,j}$,
\end{enumerate}
where $Z(K^R)$ denotes Koszul cycles, $[-]$ denotes the homology class of a cycle, and $\overline{h} = (-1)^{i+1}h$ for $h \in K_i^R$. The element $h_{1,t}:=\sum_{l=1}^t \overline{h_{1l}}h_{lt}$ is then a cycle, and its class $[h_{1,t}] \in H_{\geq 1}(K^R)$ is called a $t$-ary Massey product of $h_1,\ldots,h_t$; the set $\mu_t(h_1,\ldots,h_t)$ is the collection of all such Massey products. We say that $R=Q/I$ satisfies $(B_t)$ if all $k$-ary Massey products are defined, and contain only zero, for $k=2,\ldots,t$. It can be shown that, if $R$ satisfies $(B_{t-1})$, then all $t$-ary Massey products are defined, and contain only one element, see \cite[Proposition 2.3]{May69}. Finally, we recall that, by Golod's characterization, a ring $R$ is Golod if and only if it satisfies $(B_t)$ for all $t$ \cite{GolodMassey}. Moreover, it suffices to show that all $t$-ary Massey products vanish on a fixed choice of a $\Bbbk$-basis for $H_{\geq 1}(K^R)$.

Now let $I \subseteq Q=\Bbbk[x_1,\ldots,x_n]$ be a monomial ideal, and $R=Q/I$.

For a monomial $u = x_1^{a_1} \cdots x_n^{a_n} \in Q$, let $\bdeg(u) = (a_1,\ldots,a_n)$ denote its multidegree. Note that the Koszul complex $K^{R} \cong \bigwedge (\bigoplus_{i=1}^n Re_i)$, where $\partial(e_i)=x_i$, is multigraded by setting $\bdeg(e_i) = \bdeg(x_i)$ for all $i$. 

Choose a total order $\prec$ on $\mingens(I)$, the minimal monomial generating set of $I$. The Taylor complex $\mathbb{T}_\bullet$ is the complex of free $Q$-modules with basis $\{e_J \mid J \subseteq \mingens(I)\}$, and multi-graded with respect to $\bdeg(e_J) = \bdeg(u_J)$, where $u_j=(\lcm(u) \mid u \in J)$. The differential is given by
\[
\partial(e_J) = \sum_{u\in J} (-1)^{\sigma(u,J)} \frac{u_J}{u_{J \setminus \{u\}}} e_{J \setminus \{u\}},
\]
where $\sigma(u,J)$ is the cardinality of the set $\{v \in J \mid v \prec u\}$. One has that $\mathbb{T}_\bullet$ is a dg-algebra, with product
\[
e_{J} \cdot e_{L} = \begin{cases} (-1)^{\sigma(J,L)} \frac{u_Ju_L}{u_{J \cup L}} e_{J \cup L} & \text{ if } J \cap L = \emptyset \\
0 & \text{ otherwise,}
\end{cases}
\]
where $\sigma(J,L)$ is the cardinality of $\{(u,v) \in J \times L \mid u \prec v\}$. Further, there are quasi-isomorphisms of dg-algebras
\begin{equation}
    \label{eq Taylor}
\mathbb{T}_\bullet \otimes \Bbbk \longleftarrow \mathbb{T}_{\bullet} \otimes K^Q \longrightarrow R \otimes K^Q \cong K^{R},
\end{equation}
yielding multigraded isomorphisms $H_*(\mathbb{T}_\bullet \otimes \Bbbk) \cong \Tor_*^Q(R,\Bbbk) \cong H_*(K^{R})$, and Massey products can be computed either over $K^{R}$ or over $\mathbb{T}_\bullet \otimes \Bbbk$ (see \cite{Kat17}).

The following remark is an extension of the proof of \cite[Proposition~6]{HRW99}.

\begin{remark}\label{rem_lcm_Koszul}

The following is well-known to experts (see for example \cite{Kat17}): if $\mathbb{T}_\bullet$ denotes the Taylor resolution of $R=Q/I$ over $Q$, then 
because of (\ref{eq Taylor}) one can find a multigraded basis for the Koszul homology. 
This leads to two observations:
\begin{enumerate}
    \item the only non-trivial multidegrees in the Koszul homology are in the lcm lattice of $I$;
    \item if $h_1,\dots, h_r\in H_*(K^R)$ are multigraded elements and $h\in \mu_t(h_1,\dots, h_t)$ is a Massey product, then 
\[
\bdeg(h)= \sum_{i=1}^t \bdeg(h_i).
\]
\end{enumerate}
In particular, if $I$ is a squarefree monomial ideal, and two homogeneous elements have a non-trivial intersection of supports, then any Massey product involving them will be trivial, as its multidegree no longer lies in the lcm lattice of $I$. 
\end{remark}

The goal of this section is to use the above observation to give a complete classification for Golod cover ideals of graphs. Let $G=(V(G),E(G))$ be a finite simple graph. A \emph{vertex cover} of $G$ is a subset $V \subseteq V(G)$ such that for each edge $\{x,y\} \in E(G)$ one has $\{x,y\} \cap V \ne \emptyset$. The \emph{cover ideal} of $G$, denoted by $J(G)$, is the monomial ideal in $\Bbbk[V(G)]$ generated by $\prod_{x\in A}x$ where $A$ ranges among the minimal (with respect to inclusion) vertex covers of $G$. Note that $J(G) = \bigcap_{\{x,y\} \in E(G)} (x,y)$ and, in particular, cover ideals are Alexander dual to edge ideals (see Section \ref{Section hypergraphs} for the definition of edge ideals).

\begin{remark} \label{rem:unmixed cover}
    Observe that, as a consequence of the description $J(G) = \bigcap_{\{x,y\} \in E(G)} (x,y)$, one has that the class of cover ideals of simple graphs coincides with the class of unmixed squarefree monomial ideals of height $2$.
\end{remark}

Consider an odd cycle $G=C_{2n-1}$ where $n\geq 2$. Then it is clear that any (minimal) vertex cover of $G$ has at least $n$ vertices. In particular, this means that any two minimal generators of $J(G)$, and thus any two multidegrees in its lcm lattice, have a non-trivial intersection of supports. If $G$ contains an odd cycle as an induced subgraph, one can make similar considerations to conclude that any two minimal vertex covers have non-trivial intersection of supports. We then obtain the following.

\begin{proposition}\label{prop:odd-cycle-Golod}
    If $G$ is a graph that contains an odd cycle as an induced subgraph, then $J(G)$ is Golod.
\end{proposition}

A graph that contains no odd cycles is called a \emph{bipartite graph}. Recall that a subset of vertices of a graph is called an \emph{independent set} if no two vertices of this subset are connected by an edge. Note that the complement of an independent set is a vertex cover, and vice versa. It is well-known that a graph $G$ is bipartite if and only if we can write $V(G)=A\sqcup B$ where $A$ and $B$ are independent sets of $G$. Moreover, $G$ is called \emph{complete bipartite} if  $G$ is bipartite with $V(G)=A\sqcup B$ for some independent sets $A$ and $B$, and 
\[
E(G)=\{\{x,y\}\ : x\in A, y\in B \}\}.
\]
In this case, if we let $|A|=m$ and $|B|=n$, the graph $G$ is denoted by $K_{m,n}$. 

\begin{lemma}\label{lem:A-and-B-are-disjoint-MVC}
    Let $G$ be a connected bipartite graph with $V(G)=A\sqcup B$ for some independent sets $A$ and $B$. Then the only two disjoint minimal vertex covers of $G$ are $A$ and $B$ themselves.
\end{lemma}

\begin{proof}
    We first argue that $A$ and $B$ are minimal vertex covers of $G$. Since they are both independent sets, and they are complements to each other, they are vertex covers. On the other hand, for each vertex $x\in A$, there must be a vertex $y\in B$ such that $\{x,y\}$ is an edge of $G$, as otherwise $x$ would be an isolated vertex. Thus, no proper subset of $A$ is a vertex cover of $G$, and the same goes for $B$.

    Now let $C_1$ and $C_2$ be two disjoint minimal vertex covers of $G$. First, assume by way of contradiction that $V(G)\setminus (C_1\sqcup C_2)\neq \emptyset$, and let $y\in V(G)\setminus (C_1\sqcup C_2)$. Since $G$ is connected, $y$ must have a neighbor $x$. Since both $C_1$ and $C_2$ are vertex covers, and $y \notin C_1 \sqcup C_2$, we conclude that $x \in C_1 \cap C_2$, and this contradicts our assumption that $C_1$ and $C_2$ are disjoint. Now we can assume that $C_2=V(G)\setminus C_1$. In particular, both $C_1$ and $C_2$ are independent sets. If $C_1$ is either $A$ or $B$, then the result follows.  Now assume by contradiction that $C_1$ is neither. This implies that $C_1\cap A$ and $C_1\cap B$ are both non-empty, and their disjoint union is $C_1$ itself. Thus
    \[
    V(G)= (C_1\cap A) \sqcup (C_1\cap B) \sqcup (C_2\cap A) \sqcup (C_2\cap B),
    \]
    where these four sets are all non-empty. Observe that no vertex in $C_1\cap A$ forms an edge with a vertex in $(C_1\cap B) \sqcup (C_2\cap A)$, as $C_1$ and $A$ are both independent sets. Similarly,  no vertex in $C_1\cap B$ forms an edge with a vertex in $(C_1\cap A) \sqcup (C_2\cap B)$. Therefore, we have
    \[
    E(G) = \{(x,y)\colon x\in C_1\cap A, y\in C_2\cap B\}\ \bigsqcup \ \{(x,y)\colon x\in C_2\cap A, y\in C_1\cap B\}.
    \]
    In particular, $G$ is not connected, and this gives the desired contradiction.
\end{proof}

\begin{proposition}\label{prop:J(G)-Golod}
    If $G$ is a connected bipartite graph that is not a complete bipartite graph, then $J(G)$ is Golod.
\end{proposition}

\begin{proof}
    Since $G$ is bipartite, we can set $V(G)=A\sqcup B$ where $A$ and $B$ are independent sets. We claim that $G$ has at least three minimal vertex covers. Indeed, $A$ and $B$ are two minimal vertex covers of $G$ by \cref{lem:A-and-B-are-disjoint-MVC}. Since $G$ is not a complete bipartite graph, there exist $x\in A$ and $y\in B$ such that $\{x,y\}$ is not an edge of $G$. Then $N_G(x)\sqcup (A\setminus \{x\})$ contains neither $A$ nor $B$, and is a vertex cover of $G$. Thus there exists a third minimal vertex cover of $G$ rather than $A$ and $B$, as claimed.

    By \cref{lem:A-and-B-are-disjoint-MVC}, the only multidegrees in the lcm lattice of $J(G)$ with disjoint supports are $\prod_{x\in A}x$ and $\prod_{y\in B}y$, and the only possible such elements are the images, in the Koszul homology, of the corresponding minimal generators of $J(G)$. Therefore, the only possibly non-trivial Massey operation on $\Bbbk[V(G)]/J(G)$ is 
    \[
    \mu_2([e_{\{u\}} \otimes 1_{\Bbbk}],[e_{\{v\}} \otimes 1_{\Bbbk}]) = [(e_{\{u\}} \otimes 1_{\Bbbk}) \cdot (e_{\{u\}} \otimes 1_{\Bbbk})] = [(e_{\{u\}} \cdot e_{\{v\}}) \otimes 1_{\Bbbk}],
    \]
    where $u=\prod_{x\in A}x$, $v= \prod_{y\in B}y$, and we used $\mathbb{T}_\bullet \otimes \Bbbk$ to compute the Massey operation (see the beginning of the current section for notation and terminology). This product, however, is trivial thanks to \cite[Lemma~2.4]{Kat17}, since our assumptions guarantee that there is a third minimal generator of $J(G)$ which necessarily divides $uv$, given that $V(G) = A \sqcup B$. Since all Massey products on $\Bbbk[V(G)]/J(G)$ vanish, it is Golod.
\end{proof}

Our main result then follows easily.

\begin{theorem}\label{thm:J(G)-Golod-or-CI}
    Let $G$ be a connected graph. The following conditions are equivalent:
    \begin{enumerate}[label=(\roman*)]
        \item $J(G)$ is not Golod;
        \item $J(G)$ is a complete intersection defined by two monomials of degree at least two;
        \item $G=K_{m,n}$ with $\min\{m,n\} \geq 2$.
    \end{enumerate}
\end{theorem}

\begin{proof}
    If $G=A\sqcup B$, where $A,B$ are independent sets, is a complete bipartite graph, then $J(G)= \left(\prod_{x\in A}x, \prod_{y\in B}y\right)$  is a complete intersection. Otherwise, $J(G)$ is Golod due to \cref{prop:odd-cycle-Golod} and \cref{prop:J(G)-Golod}, as desired.
\end{proof}


This result can be extended to non-connected graphs. 

\begin{theorem}\label{thm:J(G)-Golod-or-CI-disconnected}
    Let $G$ be a graph. 
    The following conditions are equivalent:
    \begin{enumerate}[label=(\roman*)]
        \item $J(G)$ is not Golod;
        \item $J(G)$ is a complete intersection defined by two monomials of degree at least two;
        \item $G = K_{m,n} \sqcup G'$, where $\min\{m,n\} \geq 2$ and $G'$ consists of isolated vertices.
    \end{enumerate}
\end{theorem}

\begin{proof}
    Note that the cover ideal $J(G)$ is not affected by the presence of isolated vertices. Thus, we may assume that $G$ has no isolated vertices. Due to \cref{thm:J(G)-Golod-or-CI}, it suffices to show that if $G$ is not connected, then $J(G)$ is Golod. Let $G_1,\dots, G_t$, where $t\geq 2$, be the connected components of $G$. If $G$ contains an odd cycle as an induced subgraph, then the result follows from \cref{prop:odd-cycle-Golod}. Now we can assume that $G$ is bipartite, i.e., $G_i$ is connected bipartite for any $i\in [t]$. For each $i\in [t]$, set $V(G_i)=A_i\sqcup B_i$ where $A_i$ and $B_i$ are independent sets of $G_i$. We have  $J(G)=\prod_{i=1}^t J(G_i)$, and thus any minimal monomial generator of $J(G)$ is a product of those of $J(G_i)$, where $i$ ranges in $[t]$. By \cref{lem:A-and-B-are-disjoint-MVC}, if $u,v$ are monomials in the lcm lattice of $J(G)$ with disjoint supports, then $u=\prod_{i=1}^t\prod_{x\in C_i}x$, where $C_i\in \{A_i,B_i\}$, and $v= \left( \prod_{x\in V(G)}x\right)/u$. 
    As in \cref{prop:J(G)-Golod}, the only possible non-trivial Massey product on $\Bbbk[V(G)]/J(G)$ is
    \[
    \mu_2([e_{\{u\}} \otimes 1_{\Bbbk}],[e_{\{v\}} \otimes 1_{\Bbbk}]) = [(e_{\{u\}} \otimes 1_{\Bbbk}) \cdot (e_{\{u\}} \otimes 1_{\Bbbk})] = [(e_{\{u\}} \cdot e_{\{v\}}) \otimes 1_{\Bbbk}].
    \]
    However, $J(G)$ has at least two other minimal monomial generators, and both of them must divide $uv$ since $V(G)=\supp(u)\sqcup \supp(v)$. Thus, the Massey product must be trivial by \cite[Lemma~2.4]{Kat17}, and $\Bbbk[V(G)]/J(G)$ is Golod.
\end{proof}

Thanks to Remark \ref{rem:unmixed cover} we can give an equivalent formulation of Theorem~\ref{thm:J(G)-Golod-or-CI-disconnected}, which mirrors Scheja's result \cite{Scheja}:

\begin{corollary}
    If $I$ is an unmixed squarefree monomial ideal of height $2$, then $I$ is either Golod or a complete intersection.
\end{corollary}

We now give two examples to show that the above result is optimal; both have been verified with the aid of Macaulay2 \cite{M2}.
\begin{itemize}
    \item $\mathbb{Q}[x_1,x_2,x_3,x_4,x_5]/(x_1x_3,x_1x_5,x_2x_4,x_2x_5,x_3x_4x_5)$ is neither Golod nor a complete intersection, and the underlying ideal is an  unmixed squarefree monomial ideal of height~3.
    \item $\mathbb{Q}[x_1,x_2,x_3,x_4,x_5]/(x_1x_3,x_1x_4,x_2x_5)$ is neither Golod nor a complete intersection, and the underlying ideal is a squarefree monomial ideal of height $2$ that is not unmixed.
\end{itemize}

We remark that the techniques used in the proof of Theorem~\ref{thm:J(G)-Golod-or-CI} can be applied to a different class of monomial ideals, where we can also guarantee that all Massey products to be trivial. We therefore state without proofs the following result.


\begin{theorem}
    Assume that $I$ is squarefree monomial ideal in $\Bbbk[x_1,\dots, x_{2n}]$ minimally generated by monomials of degree at least $n$. Then $I$ is either Golod or a complete intersection. Moreover, $I$ is a complete intersection if and only if $I=(m_1,m_2)$ where $\deg m_1=\deg m_2=n$ and $\supp(m_1)\cap \supp(m_2)=\emptyset$.
\end{theorem}

\section{Third technique: lcm-strongly $d$-Golod monomial ideals}

As before, $\Bbbk$ is a field and $Q=\Bbbk[x_1,\ldots,x_n]$ in its standard grading, with $\m = (x_1,\ldots,x_n)$. We also let $I \subseteq Q$ be a homogeneous ideal and set $R=Q/I$.

In the case that ${\rm char}(\Bbbk)=0$, Herzog used the derivations 
\[
\partial_{x_i}=\frac{d}{dx_i}:Q\to Q\quad i=1,\ldots,n
\]
to construct canonical cycle representatives in the Koszul complex $K^R$ for the Koszul homology $H_{*}(K^R)$ \cite{Herzog91}. Later, Herzog and Maleki introduced modified operators
\[
d^i:Q\to Q\quad i=1,\ldots,n
\]
which play a similar role in arbitrary characteristic; see \cite{HM18} for their definition. These operations are not derivations, but they satisfy an (unscaled) Euler formula $f=d^1(f)x_1+\cdots + d^n(f)x_n$, and this allowed them to construct canonical cycle representatives for the Koszul homology $H_{*}(K^R)$ independent of ${\rm char}(\Bbbk)$.

\begin{definition}
    When ${\rm char}(\Bbbk)=0$, the ideal $I\subseteq Q$ is called \emph{strongly Golod} if $\partial_{x_i}(I)\partial_{x_j}(J)\subseteq I$ for all $0\leq i,j\leq n$ \cite{Herzog91}.
    
    In general characteristic, the ideal $I\subseteq Q$ is called \emph{$d$-Golod} if $d^i(I)d^j(I)\subseteq I$ for all $0\leq i,j\leq n$ \cite{HM18}. The operations $d^i$ depend on the ordering of the variables of $Q$, and so the authors of  \cite{HM18} further define $I$ to be \emph{strongly $d$-Golod} if $\sigma(I)$ is $d$-Golod for any permutation $\sigma$ of the variables $\{x_1,\ldots, x_n\}$.
\end{definition}


We note the following result, which was stated to hold in characteristic $0$ in \cite{HM18}. Their proofs in fact hold in all characteristics.

\begin{proposition}
\label{prop_sqfree_not_strongly_Golod}
    A monomial ideal $I$ is strongly $d$-Golod (or strongly Golod, in characteristic zero) if and only if for any $u,v\in \mingens(I)$, and for any variables $x_i\mid u$ and $x_j\mid v$, we have $(uv)/(x_ix_j)\in I$. In particular, nonzero squarefree monomial ideals are never strongly $d$-Golod (or strongly Golod). 
\end{proposition}
\begin{proof}
Recall from \cite{HM18} that if $u$ is a monomial and $r$ the smallest integer such that $x_r\mid u$, then $d^r(u)=u/x_r$, and $d^i(u)=0$ for any $i\neq r$.

    Suppose that $I$ is $d$-Golod. Let $u,v$ be any two monomials in $\mingens(I)$ with $x_i\mid u$ and $x_j\mid v$.  
    If $i=j$ we can choose a permutation $\sigma$ such that $\sigma(x_i)=\sigma(x_j)=x_1$.     Then $d^1\sigma(u)=\sigma(u)/x_1$ and $d^1\sigma(v)=\sigma(v)/x_1$. Since $\sigma(I)$  is $d$-Golod we get $(\sigma(u)/x_1)(\sigma(v)/x_1)=\sigma\big(uv/x_ix_j\big)\in \sigma(I)$, or equivalently, $(uv)/(x_ix_j)\in I$. Likewise if $i\neq j$ we choose $\sigma$ such that $\sigma(x_i)=x_1$ and $\sigma(x_j)=x_2$.     If $x_i|v$ then clearly $(uv)/(x_ix_j)\in I$, so we can assume $x_i\nmid v$, or equivalently $x_1\nmid \sigma(v)$. This means that $d^1\sigma(u)=\sigma(u)/x_1$ and $d^2\sigma(v)=\sigma(v)/x_2$. Since $\sigma(I)$ is $d$-Golod, we get $(\sigma(u)/x_1)(\sigma(v)/x_2)=\sigma\big(uv/x_ix_j\big)\in \sigma(I)$, or equivalently, $(uv)/(x_ix_j)\in I$.
    
    The converse implication follows directly from the definition of the $d^i$.

    For the final assertion, let $m\in \mingens(I)$. Then for any $x\mid m$, we have $(m/x)^2\in  I$ if and only if $m/x\in I$, as $I$ is a squarefree monomial ideal. By the definition of $m$, we have $(m/x)^2\notin I$, and thus $I$ is not $d$-Golod.

    The proof in characteristic zero is similar (and in fact slightly simpler).
\end{proof}

Since squarefree monomial ideals are essentially never strongly Golod, Herzog and Huneke introduced \emph{squarefree strongly Golod} ideals in \cite{HH13}: these are the ideals $I$ generated by squarefree monomials and satisfying $(\partial_{x_i}(I)\partial_{x_j}(I))^{\rm sqf} \subseteq I$ for all $i,j$. Their argument that squarefree strongly Golod ideals are Golod is flawed because it relies on \cite[Theorem 3]{BJ07}, which is known not to hold in general. We give an alternative proof of this fact:




\begin{theorem}\label{th_sqfree_strongly_golod}
    Let $I$ be a square-free monomial ideal in $Q$ such that $(\partial_{x_i}(I)\partial_{x_j}(I))^{\rm sqf} \subseteq I$ for all $i\neq j$ (in particular, this holds if $I$ is squarefree strongly Golod). Then $I$ is Golod.
\end{theorem}

This proof relies on understanding products in the \emph{reduced Koszul complex}, which has been used in toric topology to model the cohomology of moment angle complexes \cite[Construction 3.2.5]{Toric}. If $I\subseteq Q=k[x_1,\ldots,x_n]$ is a squarefree monomial ideal and $R=Q/I$, recall that $K^R=\Lambda^R(e_1,\ldots,e_n)$ with $\partial(e_i)=x_i$. The reduced Koszul complex is the quotient
\[
K_{\rm red}^R=K^R/J\quad \text{where}\quad J=(x_i^2,x_ie_i~|~ i=1,\ldots,n),
\]
in other words, $J$ is the ideal of non-squarefree elements in $K^R$. By \cite[Lemma 3.2.6]{Toric} the reduced Koszul complex is a well-defined dg algebra and the quotient map $K^R\to K_{\rm red}^R$ is a quasi-isomorphism of dg alebras. 

\begin{proof}[Proof of \cref{th_sqfree_strongly_golod}]
By \cite{Herzog91} there is a basis for the Koszul homology $H_{>0}(K^R)$ consisting of cycles $\{c_t\}_t$ in $K^R$ that lie in the ideal generated elements of $\partial_{x_i}(I)e_i$ for $i$. By \cite[Lemma 3.2.6]{Toric} the images $\{\overline{c}_t\}_t$ in $K_{\rm red}^R$ also yield a basis for the homology $H_{>0}(K_{\rm red}^R)=H_{>0}(K^R)$. By definition, $\partial_{x_i}(I)\partial_{x_j}(I)e_ie_j \subseteq (I, x_1^2,\ldots ,x_n^2)= 0$ in $K_{\rm red}^R$. Hence $\overline{c}_s\overline{c}_t=0$ for all $s$ and $t$. Since $H_{>0}(K_{\rm red}^R)$ has basis represented by cycles with pairwise zero products, it follows that $K_{\rm red}^R$ admits a trivial Massey operation, and therefore $R$ is Golod by \cite{GolodMassey}.
\end{proof}

Herzog and Huneke used strongly Golod ideals to show that all ordinary and symbolic powers of homogeneous ideals are Golod \cite{HH13}. We obtain an analog for squarefree and squarefree symbolic powers of squarefree monomial ideals. We recall that for a squarefree monomial ideal $I$ and an integer $k\geq 1$, the $k$-th \emph{squarefree power} and \emph{squarefree symbolic power} of $I$ is defined to be
\[
I^{[k]}\coloneqq \left( I^k\right)^{\rm sqf} \text{ and } I^{\{k\}}\coloneqq \left( I^{(k)}\right)^{\rm sqf},
\]
respectively. We then have the following.

\begin{theorem}
    For any squarefree monomial ideal $I$ and integer $k\geq 2$, the ideals $I^{[k]}$ and $I^{\{k\}}$ are lcm-strongly $d$-Golod, and thus Golod.
\end{theorem}

\begin{proof}
    By definition, $I^{[k]}$ is generated by monomials of the form $m_1\cdots m_k$ where $m_1,\dots,m_k\in \mingens(I)$ have pairwise disjoint supports. Consider two generators $m=m_1m_2\cdots m_k$ and $m'=m_1'm_2'\cdots m_k'$ of $I^{[k]}$ where all these $2k$ monomials have pairwise disjoint supports, and two variables $x\in \supp(m_1)$ and $x'\in \supp(m_1')$. It is clear from our description of $I^{[k]}$ above that $\frac{mm'}{xx'}$, a multiple of $(m_2\cdots m_k)(m_2'\cdots m_k')$, is in $I^{[k]}$. Therefore, $I^{[k]}$ is lcm-strongly $d$-Golod, as desired.

    Now we consider the ideal $I^{\{k\}}$. Let $I=P_1\cap \cdots \cap P_q$ be the unique prime decomposition of $I$. Then
    \[
    I^{\{k\}} = P_1^{[k]}\cap \cdots P_q^{[k]}.
    \]
    Now consider two generators $m$ and $m'$ of $I^{\{k\}}$ with disjoint support, and two variables $x\in \supp(m)$ and $x'\in \supp(m')$. Then by the same argument as in the proof for the squarefree powers, we have $\frac{mm'}{xx'}\in P_i^{[k]}$ for any $i\in [q]$, and thus it is also in $I^{\{k\}}$. Therefore, $I^{\{k\}}$ is lcm-strongly $d$-Golod, as desired.
\end{proof}


The proof of \cref{th_sqfree_strongly_golod} can be adapted to prove a stronger result covering a larger class of ideals. For a monomial ideal $I$, let $\lcm(I)$ denote the least common multiples of all minimal monomial generators of $I$.

\begin{definition}
    A monomial ideal $I$ is \emph{lcm-strongly $d$-Golod} if $\left(d^i(I)d^j(I) \right)^{\leq \lcm(I)} \subseteq I$  for any $i,j\in [n]$ where $i\neq j$. 

    When the characteristic of $\Bbbk$ is zero we do not need to use the theory of Herzog and Maleki. In this case (and only this case) we say that a monomial ideal $I$ is called \emph{lcm-strongly Golod} if $(\partial_{x_i}(I)\partial_{x_j}(I))^{\leq \lcm(I)} \subseteq I$ for each pair $i \neq j$.
\end{definition}


The lcm-strongly Golod condition is a slight generalisation of a notion introduced, under the same name, in \cite[4.4]{DESTEFANI2016}. There, the same condition $(\partial_{x_i}(I)\partial_{x_j}(I))^{\leq \lcm(I)} \subseteq I$ was required to hold for all pairs $i,j$. It was noted in \cite[Discussion after Theorem 4.6]{DESTEFANI2016} that for for understanding products in the Koszul complex it is sufficient to ask this when  $i\neq j$. For this reason we feel that it is safe to be slightly revisionist in using lcm-strongly Golod to refer to the more general notion.


\begin{theorem}\label{th_lcm_str}
    Let $I$ be a monomial ideal in $Q$. If $I$ is lcm-strongly $d$-Golod (or lcm-strongly Golod in characteristic zero) then $I$ is Golod.
\end{theorem}

To prove this we introduce the \emph{lcm-reduced Koszul complex}, defined to be the quotient $K_{\rm lcm}^R=K^R/J$ where 
\[
J= (K^R)^{\not\leq {\rm lcm}(I)}=\left(x\in K^R~|~ x \text{ is mulihomogeneous with }\bdeg(x)\not\leq {\rm lcm}(I) \right).
\]
This is analogous to the reduced Koszul complex used above, and the proof of the next lemma is similar to \cite[Lemma 3.2.6]{Toric}.

\begin{lemma}\label{lem_Koslcm}
For any monomial ring $R$ the quotient $K^R\to K_{\rm lcm}^R$ is a quasi-isomorphism of dg algebras.
\end{lemma}

\begin{proof}
It follows directly from the definition that $J$ is an ideal of $K^R$ that is closed under the differential, so $K^R\to K_{\rm lcm}^R$ is a homomorphism of dg algebras. In fact, $K^R=\bigoplus_{m\in \mathbb{N}^n} (K^R)^m$ is the direct sum of its multihomogeneous subcomplexes, and the map $K^R\to K_{\rm lcm}^R$ simply quotients by those direct summands indexed by $m\not\leq{\rm lcm}(I)$. By (1) of \cref{rem_lcm_Koszul} these summands are acyclic, and therefore $K^R\to K_{\rm lcm}^R$ is a quasi-isomorphism.
\end{proof}

The proof of \cref{th_lcm_str} is now similar to that of \cref{th_sqfree_strongly_golod}.

\begin{proof}[Proof of \cref{th_lcm_str}.]
    By \cite{HM18} there is a basis for the Koszul homology $H_{>0}(K^R)$ consisting of cycles $\{c_t\}_t$ in $K^R$ that lie in the ideal generated elements of $d^i(I)e_i$ for $i$. By \cref{lem_Koslcm} the images $\{\overline{c}_t\}_t$ in $K_{\rm lcm}^R$ also yield a basis for the homology $H_{>0}(K_{\rm lcm}^R)$. By hypothesis $d^i(I)d^j(I)e_ie_j \subseteq I + (m\ \vert\  \bdeg(m)\not\leq {\rm lcm}(I))= 0$ in $K_{\rm lcm}^R$. Hence $\overline{c}_s\overline{c}_t=0$ for all $s$ and $t$. It follows that $K_{\rm lcm}^R$ admits a trivial Massey operation, and therefore $R$ is Golod by \cite{GolodMassey}. The statement in characteristic zero has a similar proof.
\end{proof}

\begin{remark}
Let $I$ be a monomial ideal and $\pol(I)$ be its polarization. We observe that the notions that $I$ is $\lcm$-strongly Golod and $\pol(I)$ is squarefree strongly Golod are unrelated. For example, $I=(x_1^3,x_1x_2,x_1x_3)$ is not $\lcm$-strongly Golod since $x_1^2=d_2(x_1x_2)d_3(x_1x_3)$ divides $\lcm(I) = x_1^3x_2x_3$, but does not belong to $I$, whereas $\pol(I)=(x_{1,1}x_{1,2}x_{1,3},x_{1,1}x_{2,1},x_{1,1}x_{3,1})$ is squarefree strongly Golod. On the other hand, $I=(x_1^2,x_1x_2,x_2^2)$ is $\lcm$-strongly Golod, but its polarization $\pol(I)=(x_{1,1}x_{1,2},x_{1,1}x_{2,1},x_{2,1}x_{2,2})$ is not squarefree strongly Golod as $d_{1,1}(x_{1,1}x_{1,2}) d_{2,1}(x_{2,1}x_{2,2}) = x_{1,2}x_{2,2}$ is squarefree, but it does not belong to $\pol(I)$.
\end{remark}

We end this section with a brief discussion of the hypothesis in \cref{th_sqfree_strongly_golod}. Since the term ``squarefree strongly Golod'' means something slightly different in \cite{HH13}, and since we are running out of reasonable adjectives to prepend, we use geometric terminology in the definition:

\begin{definition}
        Let $k$ be a field of any characteristic. A squarefree monomial ideal is called \emph{is called face-cut Golod} if $(\partial_i(I)\partial_j(I))^{\rm sqf} \subseteq I$ for each pair $i \ne j$.
\end{definition}


\begin{proposition}
    Let $\Delta$ be a simplicial complex on $n$ vertices, and let $I$ be the corresponding monomial ideal in $Q$. Then $I$ is face-cut Golod if and only if  for every face $F\in \Delta$ and any two distinct vertices $i, j$ of $\Delta$, if $F$ is a disjoint union $U\cup V$ then at least one of $U\cup\{i\}$ or $V\cup\{j\}$ is a face of $\Delta$.
\end{proposition}

\begin{proof}
The ideal $I$ is face-cut Golod if for any two monomials $m,m'\in \mingens(I)$ and any two distinct variables $x_i,x_j$, 
the monomial 
$ (m/x_i)(m'/x_j)
$
either contains a square, or is in $I$. If $
(m/x_i)(m'/x_j)
$ is squarefree then it corresponds to a subset $F$ of $\{1,\ldots,n\}$, and $F$ is a face if and only if $
(m/x_i)(m'/x_j)
\notin I$. If we set $U$ and $V$ to be the subsets of  $\{1,\ldots,n\}$ corresponding to $m/x_i$ and $m'/x_j$, then the definition of face-cut Golod corresponds to the statement of the~proposition.
\end{proof}

\section{Golod squarefree monomial ideals and open problems} \label{Section hypergraphs}

Squarefree monomial ideals can be viewed as edge ideals of hypergraphs, a more general construction than edge ideals of graphs (cf. \cite{HVT08}).  

A \emph{hypergraph} $\mathcal{H} = (V(\mathcal{H}),E(\mathcal{H}))$ consists of a vertex set $V(\mathcal{H}) = \{x_1, \dots, x_r\}$ and an edge set $E(\mathcal{H})$, whose elements are subsets of $V(\mathcal{H})$. We restrict our attention to \emph{simple} hypergraphs; that is, when there is no nontrivial containment between edges in $E(\mathcal{H})$. A simple hypergraph is also referred to as a \emph{Sperner system}. A simple graph is a simple hypergraph in which each edge has cardinality 2.

The \emph{edge ideal} of a hypergraph $\mathcal{H}$ is constructed in a similar fashion as that of a graph:
\[
I(\mathcal{H})\coloneqq \left\langle \prod_{x\in e} x ~\middle|~ e \in E(\mathcal{H})\right\rangle \subseteq S = \Bbbk[V(\mathcal{H})].
\]

With this concept, it is easy to combinatorially describe lcm-strongly $d$-Golod squarefree monomial ideals.

\begin{lemma}\label{lem:lcm-strongly-d-Golod--squarefree}
    A squarefree monomial ideal $I=I(\mathcal{H})$ is lcm-strongly $d$-Golod if for any two disjoint edges $E_1,E_2\in E(\mathcal{H})$, the set $E_1\cup E_2 \smallsetminus \{x,y\}$, where $x\in E_1$ and $y\in E_2$, contains an edge of $\mathcal{H}$.\qed
\end{lemma}

With this new definition, it is clear that the lcm-strongly $d$-Golod property also descends along restriction, just like Golodness.

\begin{lemma}\label{lem:strongly-Golod-descends}
    Let $I$ be a lcm-strongly $d$-Golod squarefree monomial ideal. Then so is $I^{\leq m}$ for any monomial $m$. \qed
\end{lemma}

There are many types of monomial ideals associated with graphs in the literature, e.g., edge ideals, path ideals, or connected ideals. It is a problem of interest to characterize graphs whose associated ideals have a certain algebraic property. The most classical such result is arguably Fr\"oberg's theorem: a graph is co-chordal if and only if its edge ideal has a linear resolution. A popular technique to approach an analog is to check whether the property is inherited by a (induced) subgraph. Recall that $\mathcal{H}'= (V(\mathcal{H}'),E(\mathcal{H}'))$ is an \emph{induced sub-hypergraph}  of a hypergraph $\mathcal{H}= (V(\mathcal{H}),E(\mathcal{H}))$  if the following holds:
\begin{itemize}
    \item $V(\mathcal{H}')\subseteq V(\mathcal{H})$,\ $E(\mathcal{H}')\subseteq E(\mathcal{H})$;
    \item if for some integer $t$ and vertices $x_{i_1},\dots, x_{i_t}\in V(\mathcal{H}')$, we have $\{x_{i_1},\dots, x_{i_t}\}\in E(\mathcal{H})$, then $\{x_{i_1},\dots, x_{i_t}\}\in E(\mathcal{H}')$.
\end{itemize}

The properties of Golodness and lcm-strong Golodness descend to induced sub-hypergraphs:

\begin{lemma}\label{lem:Golod-induced-subgraph}
    Let $\mathcal{H'}$ be an induced sub-hypergraph of a hypergraph $\mathcal{H}$. If $I(\mathcal{H})$ is Golod (resp, lcm-strongly Golod), then so is $I(\mathcal{H}')$. 
\end{lemma}

\begin{proof}
    It follows directly from definition that
    \[
    I(\mathcal{H}')=\left( I(\mathcal{H}) \right)^{\leq \prod_{x\in V(\mathcal{H}')} x}.
    \]
    The result then follows from Theorem~\ref{thm:Golod-restriction} and  Lemma~\ref{lem:strongly-Golod-descends}.
\end{proof}

Characterizing all hypergraphs whose edge ideal is Golod is a problem of interest, especially in toric topology, where the Golod property for Stanley--Reisner ideals corresponds to geometric properties of moment-angle spaces. See for example \cite{Amelotte2024,BJ07,Frank2018,IK2018,Kat16}.

Due to Lemma~\ref{lem:Golod-induced-subgraph}, both the properties of Golodness and lcm-strong Golodness are inherited by induced sub-hypergraphs, and thus the above problem becomes a purely combinatorial one: finding (all) forbidden structures. We will give an example in the case of edge ideals of graphs. Recall that if $G$ is a graph, then $G^c$ is called the \emph{complement graph} of $G$, with the vertex and edge sets as follows:
\[
V(G^c)=V(G), \quad E(G^c) =\big\{ \{x,y\}\colon x,y\in V(G)  \big\}\setminus E(G).
\]
All Golod edge ideals have already been classified.

\begin{theorem}[\protect{\cite{CINR15}}]\label{thm:Golod-edge-ideals}
    Let $G$ be a finite simple graph with no isolated vertices. Then $I(G)$ is Golod if and only if $G^c$ does not contain the $n$-cycle graph for any $n\geq 4$, as an induced subgraph.
\end{theorem}

Next we classify all lcm-strongly Golod edge ideals.
Recall that the complete bipartite graph $K_{1,n-1}$ is also known as the \emph{star graph} $G$ on $n$ vertices.

\begin{theorem}\label{thm:lcm-strongly-Golod-edge-ideals}
    Let $G$ be a finite simple graph with no isolated vertices. The following are equivalent:
    \begin{enumerate}
        \item $I(G)$ is lcm-strongly Golod.
        \item The only induced subgraph of $G$ that has four vertices and two disjoint edges is the complete graph $K_4$ \Kfoursymb[1].
        \item $G$ does not contain the union of two disjoint edges \twoptwosymbol[1], the 4-path graph \pfoursymbol[1], the paw graph \pawsymb[1], the 4-cycle graph \cfoursymbol[1], or the diamond graph \diamondsymb[1], as an induced subgraph. 
        \item $G$ is either a star graph or a complete graph.
    \end{enumerate}
\end{theorem}

\begin{proof}
    \emph{(1) $\Longleftrightarrow$ (2):} 
    For convenience, we first rephrase the definition of being lcm-strongly Golod for the ideals of interest: the ideal $I(G)$ is lcm-strongly Golod if and only for any two disjoint edges $\{w,x\}$ and $\{y,z\}$ of $G$,  the four sets $\{w,y\},\{w,z\},\{x,y\},$ and $\{x,z\}$ each contains an edge of $G$. In our case, this forces these four sets to be edges of $G$. In other words, the ideal $I(G)$ is lcm-strongly Golod if and only if the only induced subgraph of $G$ that has four vertices and two disjoint edges is the complete graph $K_4$, as desired.

    \emph{(2) $\implies$ (3):} The union of two disjoint edges \twoptwosymbol[1], the 4-path graph \pfoursymbol[1], the paw graph \pawsymb[1], the 4-cycle graph \cfoursymbol[1], and the diamond graph \diamondsymb[1] do not satisfy (2). Since (2) is a condition that is inheritted by induced subgraphs, (3) follows.

    \emph{(3) $\implies$ (2):} This is straightforward as the only graph with four vertices and two disjoint edges are the five listed in (3), together with the complete graph $K_4$ \Kfoursymb[1]. Thus (2) follows from~(3).
    
    \emph{(2)+(3) $\implies$ (4):} As $G$ does not contain the union of two disjoint edges \twoptwosymbol[1], it is connected. Suppose that $G$ is a tree, i.e., $G$ does not contain any induced cycle. Then $G$ is a star graph since $G$ does not contain any induced 4-path graph \pfoursymbol[1], as desired. 
    
    Now we can assume that $G$ contains an induced cycle. Since $G$ does not contain any induced 4-path \pfoursymbol[1], it does not contain any induced $n$-cycle, either, for any $n\geq 5$. Also recall that $G$ does not contain any induced 4-cycle \cfoursymbol[1] by (3). Therefore $G$ contains an induced 3-cycle, which we assume to be formed by three vertices $x_1,x_2,x_3$. 
    
    Set $V(G)=\{x_1,x_2,\dots, x_n\}$ such that for any $i\in [2,n]$, the vertex $x_i$ is incident to some other vertex $x_j$ where $j\in [1,i-1]$. Note that we have $n\geq 3$. We have the following claim.
    \begin{claim}
        For any $i\in [3,n]$, the induced subgraph of $G$ formed by $x_1,\dots, x_i$ is the complete graph $K_i$.
    \end{claim}
    \begin{proof}
        We have already proved the base case $i=3$, as $x_1,x_2,x_3$ already form an induced 3-cycle of $G$. By induction, assume that $i\geq 4$, and the induced subgraph of $G$ formed by $x_1,\dots, x_{i-1}$ is the complete graph $K_{i-1}$. It now suffices to show that $x_i$ is incident to $x_1,\dots, x_{i-1}$. By our assumption, $x_i$ is incident to some $x_j$ where $j\in [1,i-1]$. Without loss of generality assume that $j=i-1$. For any $k\in [2,i-2]$, consider the induced subgraph of $G$ formed by $x_1,x_k,x_{i-1},x_i$. This induced subgraph has four vertices and two disjoint edges ($x_1x_k$ by the induction hypothesis, and $x_{i-1}x_i$), and therefore must be the complete graph $K_{4}$ \Kfoursymb[1], by (2). In particular, this implies that $x_i$ is incident to $x_1$ and $x_k$ for any $k\in [2,i-2]$. This concludes the~proof.
    \end{proof}

    (4) then follows from the claim when $i=n$.
    
    \emph{(4) $\implies$ (3):} This is straightforward from the definition of star and complete graphs. 
\end{proof}

An explicit description of all hypergraphs whose edge ideal is Golod seems out-of-reach at this moment. A description of all hypergraphs whose edge ideal is lcm-strongly Golod in terms of forbidden induced sub-hypergraphs can be theoretically obtained from Lemma~\ref{lem:lcm-strongly-d-Golod--squarefree}, a task that is easier with more imposed conditions. Theorem~\ref{thm:lcm-strongly-Golod-edge-ideals} (3) is a complete list of forbidden induced sub-hypergraphs when the hypergraph in question is assumed to be \emph{$2$-uniform}, i.e., every edge is of cardinality two. Taking path ideals or connected ideals of a graph is a generalization of taking edge ideals, and each such action corresponds to a special hypergraph constructed from a finite simple graph. In this direction, one may ask whether Theorem~\ref{thm:lcm-strongly-Golod-edge-ideals} can be generalized accordingly for such hypergraphs, leading to the following:

\begin{question} \label{quest:lcm strongly Golod path and connected}
    For each integer $t\geq 3$, what are all finite simple graphs that admit a lcm-strongly Golod $t$-path ideal/$t$-connected ideal?
\end{question}

It is important to note that, under suitable conditions (e.g., finite simple graphs), the full family of forbidden induced sub-hypergraphs for the property of Golodness is an infinite one, while that for the property of lcm-strong Golodness is finite. The latter can be seen from Lemma~\ref{lem:lcm-strongly-d-Golod--squarefree}. This phenomenon carries over to the cases of path ideals and connected ideals, confirming that they offer a natural extension to the class of edge ideals to be studied in terms of Question~\ref{quest:lcm strongly Golod path and connected}.

\bibliographystyle{amsplain}
\bibliography{refs}

@Misc{M2,
          author = {Grayson, Daniel R. and Stillman, Michael E.},
          title = {Macaulay2, a software system for research in algebraic geometry},
          howpublished = {Available at \url{http://www2.macaulay2.com}}
        }

@article {May69,
    AUTHOR = {May, J. Peter},
     TITLE = {Matric {M}assey products},
   JOURNAL = {J. Algebra},
  FJOURNAL = {Journal of Algebra},
    VOLUME = {12},
      YEAR = {1969},
     PAGES = {533--568},
      ISSN = {0021-8693},
   MRCLASS = {18.20 (55.00)},
  MRNUMBER = {238929},
       DOI = {10.1016/0021-8693(69)90027-1},
       URL = {https://doi.org/10.1016/0021-8693(69)90027-1},
}

@article {RossiSega2014,
    AUTHOR = {Rossi, Maria Evelina and \c Sega, Liana M.},
     TITLE = {Poincar\'e{} series of modules over compressed {G}orenstein
              local rings},
   JOURNAL = {Adv. Math.},
  FJOURNAL = {Advances in Mathematics},
    VOLUME = {259},
      YEAR = {2014},
     PAGES = {421--447},
      ISSN = {0001-8708,1090-2082},
   MRCLASS = {13D02 (13A02 13D07 13E10 13H10)},
  MRNUMBER = {3197663},
MRREVIEWER = {Lars\ Winther\ Christensen},
       DOI = {10.1016/j.aim.2014.03.024},
       URL = {https://doi.org/10.1016/j.aim.2014.03.024},
}

@article {VdB2022,
    AUTHOR = {VandeBogert, Keller},
     TITLE = {Products of ideals and {G}olod rings},
   JOURNAL = {Proc. Amer. Math. Soc.},
  FJOURNAL = {Proceedings of the American Mathematical Society},
    VOLUME = {150},
      YEAR = {2022},
    NUMBER = {8},
     PAGES = {3345--3356},
      ISSN = {0002-9939,1088-6826},
   MRCLASS = {13D02 (13C13 13D07)},
  MRNUMBER = {4439458},
MRREVIEWER = {Thanh\ Vu},
       DOI = {10.1090/proc/15968},
       URL = {https://doi.org/10.1090/proc/15968},
}

@article {Levin1975,
    AUTHOR = {Levin, Gerson},
     TITLE = {Local rings and {G}olod homomorphisms},
   JOURNAL = {J. Algebra},
  FJOURNAL = {Journal of Algebra},
    VOLUME = {37},
      YEAR = {1975},
    NUMBER = {2},
     PAGES = {266--289},
      ISSN = {0021-8693},
   MRCLASS = {13H10},
  MRNUMBER = {429868},
MRREVIEWER = {Tadayuki\ Matsuoka},
       DOI = {10.1016/0021-8693(75)90077-0},
       URL = {https://doi.org/10.1016/0021-8693(75)90077-0},
}

@article {DDS2022,
    AUTHOR = {Dao, Hailong and De Stefani, Alessandro},
     TITLE = {On monomial {G}olod ideals},
   JOURNAL = {Acta Math. Vietnam.},
  FJOURNAL = {Acta Mathematica Vietnamica},
    VOLUME = {47},
      YEAR = {2022},
    NUMBER = {1},
     PAGES = {359--367},
      ISSN = {0251-4184,2315-4144},
   MRCLASS = {13D02 (05E40)},
  MRNUMBER = {4406577},
MRREVIEWER = {Ali\ Akbar\ Yazdan Pour},
       DOI = {10.1007/s40306-020-00390-2},
       URL = {https://doi.org/10.1007/s40306-020-00390-2},
}

@article {BJ07,
    AUTHOR = {Berglund, Alexander and J\"ollenbeck, Michael},
     TITLE = {On the {G}olod property of {S}tanley-{R}eisner rings},
   JOURNAL = {J. Algebra},
  FJOURNAL = {Journal of Algebra},
    VOLUME = {315},
      YEAR = {2007},
    NUMBER = {1},
     PAGES = {249--273},
      ISSN = {0021-8693,1090-266X},
   MRCLASS = {13F55 (13D07 13D40)},
  MRNUMBER = {2344344},
MRREVIEWER = {Christopher\ A.\ Francisco},
       DOI = {10.1016/j.jalgebra.2007.04.018},
       URL = {https://doi.org/10.1016/j.jalgebra.2007.04.018},
}

@incollection {Herzog91,
    AUTHOR = {Herzog, J\"urgen},
     TITLE = {Canonical {K}oszul cycles},
 BOOKTITLE = {International {S}eminar on {A}lgebra and its {A}pplications
              ({S}panish) ({M}\'exico {C}ity, 1991)},
    SERIES = {Aportaciones Mat. Notas Investigaci\'on},
    VOLUME = {6},
     PAGES = {33--41},
 PUBLISHER = {Soc. Mat. Mexicana, M\'exico},
      YEAR = {1992},
      ISBN = {968-36-2796-X},
   MRCLASS = {13D25 (13N05)},
  MRNUMBER = {1310371},
MRREVIEWER = {Jaime-Luis\ Garcia-Roig},
}

@article {BF85,
    AUTHOR = {Backelin, J\"orgen and Fr\"oberg, Ralf},
     TITLE = {Koszul algebras, {V}eronese subrings and rings with linear
              resolutions},
   JOURNAL = {Rev. Roumaine Math. Pures Appl.},
  FJOURNAL = {Acad\'emie de la R\'epublique Populaire Roumaine. Revue
              Roumaine de Math\'ematiques Pures et Appliqu\'ees},
    VOLUME = {30},
      YEAR = {1985},
    NUMBER = {2},
     PAGES = {85--97},
      ISSN = {0035-3965},
   MRCLASS = {16A03},
  MRNUMBER = {789425},
MRREVIEWER = {Freddy\ M. J. Van Oystaeyen},
}

@article {HRW99,
    AUTHOR = {Herzog, J. and Reiner, V. and Welker, V.},
     TITLE = {Componentwise linear ideals and {G}olod rings},
   JOURNAL = {Michigan Math. J.},
  FJOURNAL = {Michigan Mathematical Journal},
    VOLUME = {46},
      YEAR = {1999},
    NUMBER = {2},
     PAGES = {211--223},
      ISSN = {0026-2285,1945-2365},
   MRCLASS = {13F55 (13C14)},
  MRNUMBER = {1704158},
MRREVIEWER = {Ralf\ Fr\"oberg},
       DOI = {10.1307/mmj/1030132406},
       URL = {https://doi.org/10.1307/mmj/1030132406},
}

@article {GolodMassey,
    AUTHOR = {Golod, E. S.},
     TITLE = {Homologies of some local rings},
   JOURNAL = {Dokl. Akad. Nauk SSSR},
  FJOURNAL = {Doklady Akademii Nauk SSSR},
    VOLUME = {144},
      YEAR = {1962},
     PAGES = {479--482},
      ISSN = {0002-3264},
   MRCLASS = {13.95 (13.90)},
  MRNUMBER = {138667},
}

@article{DESTEFANI2016,
title = {Products of ideals may not be Golod},
journal = {Journal of Pure and Applied Algebra},
volume = {220},
number = {6},
pages = {2289-2306},
year = {2016},
issn = {0022-4049},
doi = {https://doi.org/10.1016/j.jpaa.2015.11.007},
url = {https://www.sciencedirect.com/science/article/pii/S0022404915003230},
author = {Alessandro {De Stefani}}
}

@article{CINR15,
author = {Conca, Aldo and Iyengar, Srikanth B. and  Nguyen, Hop Dang and Romer, Tim},
journal = {  Acta Math Vietnam},
pages = {353–374},
publisher = {},
title = {ABSOLUTELY KOSZUL ALGEBRAS AND
 THE BACKELIN-ROOS PROPERTY},
url = {https://doi.org/10.1007/s40306-015-0125-0},
volume = {40},
year = {2015},
}

@article{GUPTA2017,
title = {Ascent and descent of the Golod property along algebra retracts},
journal = {Journal of Algebra},
volume = {480},
pages = {124-143},
year = {2017},
issn = {0021-8693},
doi = {https://doi.org/10.1016/j.jalgebra.2017.02.009},
url = {https://www.sciencedirect.com/science/article/pii/S0021869317301084},
author = {Anjan Gupta}
}

@article {HHZ04Dirac,
    AUTHOR = {Herzog, J\"urgen and Hibi, Takayuki and Zheng, Xinxian},
     TITLE = {Dirac's theorem on chordal graphs and {A}lexander duality},
   JOURNAL = {European J. Combin.},
  FJOURNAL = {European Journal of Combinatorics},
    VOLUME = {25},
      YEAR = {2004},
    NUMBER = {7},
     PAGES = {949--960},
      ISSN = {0195-6698,1095-9971},
   MRCLASS = {05C38 (13F55 52B20)},
  MRNUMBER = {2083448},
MRREVIEWER = {Andrew\ Vince},
       DOI = {10.1016/j.ejc.2003.12.008},
       URL = {https://doi.org/10.1016/j.ejc.2003.12.008},
}

@article{BPS98,
  title={Monomial resolutions},
  author={Bayer, Dave and Peeva, Irena and Sturmfels, Bernd},
  year={1998},
  journal={ Math. Res. Lett. },
  volume = {5},
  pages = {no. 1--2, 31--46},
}

@article{OHH00,
author = {Ohsugi, Hidefumi and Herzog, Jürgen and Hibi, Takayuki},
year = {2000},
month = {09},
pages = {},
title = {Combinatorial pure subrings},
volume = {37},
journal = {Osaka Journal of Mathematics}
}

@article{HH13,
title = {Ordinary and symbolic powers are Golod},
journal = {Advances in Mathematics},
volume = {246},
pages = {89-99},
year = {2013},
issn = {0001-8708},
doi = {https://doi.org/10.1016/j.aim.2013.07.002},
url = {https://www.sciencedirect.com/science/article/pii/S0001870813002405},
author = {Jürgen Herzog and Craig Huneke}
}

@book{Toric,
  title={Toric Topology},
  author={Buchstaber, Victor M.  and Panov, Taras E. },
  year={2015},
  publisher={Mathematical Surveys and Monographs, 204, American Mathematical Society, Providence, RI}
}

@article {IK2018,
    AUTHOR = {Iriye, Kouyemon and Kishimoto, Daisuke},
     TITLE = {Golodness and polyhedral products for two-dimensional
              simplicial complexes},
   JOURNAL = {Forum Math.},
  FJOURNAL = {Forum Mathematicum},
    VOLUME = {30},
      YEAR = {2018},
    NUMBER = {2},
     PAGES = {527--532},
      ISSN = {0933-7741,1435-5337},
   MRCLASS = {55P15 (13F55 57Q15)},
  MRNUMBER = {3770008},
MRREVIEWER = {Anthony\ P.\ Bahri},
       DOI = {10.1515/forum-2017-0130},
       URL = {https://doi.org/10.1515/forum-2017-0130},
}

@article {Kat16,
    AUTHOR = {Katth\"an, Lukas},
     TITLE = {The {G}olod property for {S}tanley-{R}eisner rings in varying
              characteristic},
   JOURNAL = {J. Pure Appl. Algebra},
  FJOURNAL = {Journal of Pure and Applied Algebra},
    VOLUME = {220},
      YEAR = {2016},
    NUMBER = {6},
     PAGES = {2265--2276},
      ISSN = {0022-4049,1873-1376},
   MRCLASS = {13F55 (05E40)},
  MRNUMBER = {3448795},
MRREVIEWER = {Siamak\ Yassemi},
       DOI = {10.1016/j.jpaa.2015.11.005},
       URL = {https://doi.org/10.1016/j.jpaa.2015.11.005},
}

@Inbook{Amelotte2024,
author="Amelotte, Steven",
editor="Bahri, Anthony
and Jeffrey, Lisa
and Panov, Taras
and Stanley, Donald
and Theriault, Stephen",
title="Connected Sums of Sphere Products and Minimally Non-Golod Complexes",
bookTitle="Toric Topology and Polyhedral Products",
year="2024",
publisher="Springer Nature Switzerland",
address="Cham",
pages="1--12",
isbn="978-3-031-57204-3",
doi="10.1007/978-3-031-57204-3_1",
url="https://doi.org/10.1007/978-3-031-57204-3_1"
}

@article {Frank2018,
    AUTHOR = {Frankhuizen, Robin},
     TITLE = {{$A_\infty$}-resolutions and the {G}olod property for monomial
              rings},
   JOURNAL = {Algebr. Geom. Topol.},
  FJOURNAL = {Algebraic \& Geometric Topology},
    VOLUME = {18},
      YEAR = {2018},
    NUMBER = {6},
     PAGES = {3403--3424},
      ISSN = {1472-2747,1472-2739},
   MRCLASS = {13D07 (13D40 16E45 55S30)},
  MRNUMBER = {3868225},
MRREVIEWER = {Timothy\ B. P. Clark},
       DOI = {10.2140/agt.2018.18.3403},
       URL = {https://doi.org/10.2140/agt.2018.18.3403},
}

@book{Swan06,
  title={Polarization (of monomial ideals)},
  author={Irena Swanson},
  year={2006},
  publisher={\url{https://www.math.purdue.edu/~iswanso/polarization.pdf}}
}

@Inbook{Avramov1998,
author="Avramov, Luchezar L.",
editor="Elias, J.
and Giral, J. M.
and Mir{\'o}-Roig, R. M.
and Zarzuela, S.",
title="Infinite Free Resolutions",
bookTitle="Six Lectures on Commutative Algebra",
year="1998",
publisher="Birkh{\"a}user Basel",
address="Basel",
pages="1--118",
isbn="978-3-0346-0329-4",
doi="10.1007/978-3-0346-0329-4_1",
url="https://doi.org/10.1007/978-3-0346-0329-4_1"
}

@article {HM18,
    AUTHOR = {Herzog, J\"urgen and Maleki, Rasoul Ahangari},
     TITLE = {Koszul cycles and {G}olod rings},
   JOURNAL = {Manuscripta Math.},
  FJOURNAL = {Manuscripta Mathematica},
    VOLUME = {157},
      YEAR = {2018},
    NUMBER = {3-4},
     PAGES = {483--495},
      ISSN = {0025-2611,1432-1785},
   MRCLASS = {13A02 (13D02 13D40)},
  MRNUMBER = {3858414},
MRREVIEWER = {Raheleh\ Jafari},
       DOI = {10.1007/s00229-017-0997-5},
       URL = {https://doi.org/10.1007/s00229-017-0997-5},
}

@article {EN14,
    AUTHOR = {Epstein, Neil and Nguyen, Hop D.},
     TITLE = {Algebra retracts and {S}tanley-{R}eisner rings},
   JOURNAL = {J. Pure Appl. Algebra},
  FJOURNAL = {Journal of Pure and Applied Algebra},
    VOLUME = {218},
      YEAR = {2014},
    NUMBER = {9},
     PAGES = {1665--1682},
      ISSN = {0022-4049,1873-1376},
   MRCLASS = {13F55 (05E40 13A02 13D99 13H10)},
  MRNUMBER = {3188864},
MRREVIEWER = {Ralf\ Fr\"oberg},
       DOI = {10.1016/j.jpaa.2014.01.006},
       URL = {https://doi.org/10.1016/j.jpaa.2014.01.006},
}

@book {BCRV,
    AUTHOR = {Bruns, Winfried and Conca, Aldo and Raicu, Claudiu and
              Varbaro, Matteo},
     TITLE = {Determinants, {G}r\"obner bases and cohomology},
    SERIES = {Springer Monographs in Mathematics},
 PUBLISHER = {Springer, Cham},
      YEAR = {[2022] \copyright 2022},
     PAGES = {xiii+507},
      ISBN = {978-3-031-05479-2; 978-3-031-05480-8},
   MRCLASS = {13P10 (13C40)},
  MRNUMBER = {4627943},
       DOI = {10.1007/978-3-031-05480-8},
       URL = {https://doi.org/10.1007/978-3-031-05480-8},
}

@article {Woodroofe-vertex-decomposable,
    AUTHOR = {Woodroofe, Russ},
     TITLE = {Vertex decomposable graphs and obstructions to shellability},
   JOURNAL = {Proc. Amer. Math. Soc.},
  FJOURNAL = {Proceedings of the American Mathematical Society},
    VOLUME = {137},
      YEAR = {2009},
    NUMBER = {10},
     PAGES = {3235--3246},
      ISSN = {0002-9939,1088-6826},
   MRCLASS = {05E45 (05C38 13F55)},
  MRNUMBER = {2515394},
MRREVIEWER = {Siamak\ Yassemi},
       DOI = {10.1090/S0002-9939-09-09981-X},
       URL = {https://doi.org/10.1090/S0002-9939-09-09981-X},
}

@article {FVT-sequentially-CM,
    AUTHOR = {Francisco, Christopher A. and Van Tuyl, Adam},
     TITLE = {Sequentially {C}ohen-{M}acaulay edge ideals},
   JOURNAL = {Proc. Amer. Math. Soc.},
  FJOURNAL = {Proceedings of the American Mathematical Society},
    VOLUME = {135},
      YEAR = {2007},
    NUMBER = {8},
     PAGES = {2327--2337},
      ISSN = {0002-9939,1088-6826},
   MRCLASS = {13F55 (05E25)},
  MRNUMBER = {2302553},
MRREVIEWER = {Carles\ Bivi\`a-Ausina},
       DOI = {10.1090/S0002-9939-07-08841-7},
       URL = {https://doi.org/10.1090/S0002-9939-07-08841-7},
}

@incollection {GP81,
    AUTHOR = {Gruson, Laurent and Peskine, Christian},
     TITLE = {Courbes de l'espace projectif: vari\'et\'es de s\'ecantes},
 BOOKTITLE = {Enumerative geometry and classical algebraic geometry ({N}ice,
              1981)},
    SERIES = {Progr. Math.},
    VOLUME = {24},
     PAGES = {1--31},
 PUBLISHER = {Birkh\"auser, Boston, MA},
      YEAR = {1982},
      ISBN = {3-7643-3106-2},
   MRCLASS = {14M15 (14H45 14M05 32L25)},
  MRNUMBER = {685761},
MRREVIEWER = {Markus\ Brodmann},
}

@article {DK24,
    AUTHOR = {Deng, Jiahe and Kretschmer, Andreas},
     TITLE = {Ideals of submaximal minors of sparse symmetric matrices},
   JOURNAL = {J. Pure Appl. Algebra},
  FJOURNAL = {Journal of Pure and Applied Algebra},
    VOLUME = {228},
      YEAR = {2024},
    NUMBER = {6},
     PAGES = {Paper No. 107595, 16},
      ISSN = {0022-4049,1873-1376},
   MRCLASS = {13C40 (13D02 13P10 14N10)},
  MRNUMBER = {4695685},
MRREVIEWER = {Giuseppe\ Favacchio},
       DOI = {10.1016/j.jpaa.2023.107595},
       URL = {https://doi.org/10.1016/j.jpaa.2023.107595},
}

@article {BCMV25,
    AUTHOR = {Ada Boralevi and Enrico Carlini and Mateusz Micha{\l}ek and Emanuele Ventura},
     TITLE = {On the codimension of permanental varieties},
   JOURNAL = {Adv. Math.},
  FJOURNAL = {Advances in Mathematics},
    VOLUME = {461},
      YEAR = {2025},
     PAGES = {Paper No. 110079, 28},
      ISSN = {0001-8708,1090-2082},
   MRCLASS = {14M12 (05E14 05E40 15A15)},
       DOI = {10.1016/j.aim.2024.110079},
       URL = {https://doi.org/10.1016/j.aim.2024.110079},
}

@article{LS00,
  title={Permanental ideals},
  author={Laubenbacher, Reinhard C. and Swanson, Irena},
  year={2000},
  journal={J. Symbolic Comput.   },
  volume={30},
  pages={195--205}
}

@article {HVT08,
    AUTHOR = {H\`a, Huy T\`ai and Van Tuyl, Adam},
     TITLE = {Monomial ideals, edge ideals of hypergraphs, and their graded
              {B}etti numbers},
   JOURNAL = {J. Algebraic Combin.},
  FJOURNAL = {Journal of Algebraic Combinatorics. An International Journal},
    VOLUME = {27},
      YEAR = {2008},
    NUMBER = {2},
     PAGES = {215--245},
      ISSN = {0925-9899,1572-9192},
   MRCLASS = {05C65 (05C25 13D02)},
  MRNUMBER = {2375493},
MRREVIEWER = {Sara\ Faridi},
       DOI = {10.1007/s10801-007-0079-y},
       URL = {https://doi.org/10.1007/s10801-007-0079-y},
}

@article {EHH11,
    AUTHOR = {Ene, Viviana and Herzog, J\"urgen and Hibi, Takayuki},
     TITLE = {Cohen-{M}acaulay binomial edge ideals},
   JOURNAL = {Nagoya Math. J.},
  FJOURNAL = {Nagoya Mathematical Journal},
    VOLUME = {204},
      YEAR = {2011},
     PAGES = {57--68},
      ISSN = {0027-7630,2152-6842},
   MRCLASS = {13F20 (05C25 05E40 13C14 13D02)},
  MRNUMBER = {2863365},
MRREVIEWER = {Adam\ L.\ Van Tuyl},
       DOI = {10.1215/00277630-1431831},
       URL = {https://doi.org/10.1215/00277630-1431831},
}

@article {Kat17,
    AUTHOR = {Katth\"an, Lukas},
     TITLE = {A non-{G}olod ring with a trivial product on its {K}oszul
              homology},
   JOURNAL = {J. Algebra},
  FJOURNAL = {Journal of Algebra},
    VOLUME = {479},
      YEAR = {2017},
     PAGES = {244--262},
      ISSN = {0021-8693,1090-266X},
   MRCLASS = {13D02 (05E40 13F55)},
  MRNUMBER = {3627285},
MRREVIEWER = {Timothy\ B. P. Clark},
       DOI = {10.1016/j.jalgebra.2017.01.042},
       URL = {https://doi.org/10.1016/j.jalgebra.2017.01.042},
}

@article {Betti-cover-ideals,
    AUTHOR = {H\`a, T\`ai Huy and Hibi, Takayuki},
     TITLE = {Resolution and {B}etti numbers of vertex cover ideals},
   JOURNAL = {Math. Scand.},
  FJOURNAL = {Mathematica Scandinavica},
    VOLUME = {131},
      YEAR = {2025},
    NUMBER = {3},
     PAGES = {451--463},
      ISSN = {0025-5521,1903-1807},
   MRCLASS = {13A70},
  MRNUMBER = {4982887},
}

@article {HS15-cover-ideals,
    AUTHOR = {H\`a, Huy T\`ai and Sun, Mengyao},
     TITLE = {Squarefree monomial ideals that fail the persistence property
              and non-increasing depth},
   JOURNAL = {Acta Math. Vietnam.},
  FJOURNAL = {Acta Mathematica Vietnamica},
    VOLUME = {40},
      YEAR = {2015},
    NUMBER = {1},
     PAGES = {125--137},
      ISSN = {0251-4184,2315-4144},
   MRCLASS = {13F20 (05C15 05C25 13C15)},
  MRNUMBER = {3331937},
MRREVIEWER = {Monica\ La Barbiera},
       DOI = {10.1007/s40306-014-0104-x},
       URL = {https://doi.org/10.1007/s40306-014-0104-x},
}

@article {cover-ideals-counter,
    AUTHOR = {Kaiser, Tom\'a\u{s} and Stehl\'ik, Mat\v{e}j and \v{S}krekovski,
              Riste},
     TITLE = {Replication in critical graphs and the persistence of monomial
              ideals},
   JOURNAL = {J. Combin. Theory Ser. A},
  FJOURNAL = {Journal of Combinatorial Theory. Series A},
    VOLUME = {123},
      YEAR = {2014},
     PAGES = {239--251},
      ISSN = {0097-3165,1096-0899},
   MRCLASS = {05C15 (05C10 13A99)},
  MRNUMBER = {3157809},
MRREVIEWER = {John\ J.\ Watkins},
       DOI = {10.1016/j.jcta.2013.12.005},
       URL = {https://doi.org/10.1016/j.jcta.2013.12.005},
}

@article {CPSTY,
    AUTHOR = {Constantinescu, A. and Pournaki, M. R. and Seyed Fakhari, S.
              A. and Terai, N. and Yassemi, S.},
     TITLE = {Cohen-{M}acaulayness and limit behavior of depth for powers of
              cover ideals},
   JOURNAL = {Comm. Algebra},
  FJOURNAL = {Communications in Algebra},
    VOLUME = {43},
      YEAR = {2015},
    NUMBER = {1},
     PAGES = {143--157},
      ISSN = {0092-7872,1532-4125},
   MRCLASS = {13H10 (13C15 13F20)},
  MRNUMBER = {3240410},
MRREVIEWER = {Marco\ Fontana},
       DOI = {10.1080/00927872.2014.897550},
       URL = {https://doi.org/10.1080/00927872.2014.897550},
}

@article {ER98,
    AUTHOR = {Eagon, John A. and Reiner, Victor},
     TITLE = {Resolutions of {S}tanley-{R}eisner rings and {A}lexander
              duality},
   JOURNAL = {J. Pure Appl. Algebra},
  FJOURNAL = {Journal of Pure and Applied Algebra},
    VOLUME = {130},
      YEAR = {1998},
    NUMBER = {3},
     PAGES = {265--275},
      ISSN = {0022-4049,1873-1376},
   MRCLASS = {13D25 (13D02 13H10)},
  MRNUMBER = {1633767},
MRREVIEWER = {Ralf\ Fr\"oberg},
       DOI = {10.1016/S0022-4049(97)00097-2},
       URL = {https://doi.org/10.1016/S0022-4049(97)00097-2},
}

@article {HvT22,
    AUTHOR = {H\`a, Huy T\`ai and Van Tuyl, Adam},
     TITLE = {Powers of componentwise linear ideals: the
              {H}erzog-{H}ibi-{O}hsugi conjecture and related problems},
   JOURNAL = {Res. Math. Sci.},
  FJOURNAL = {Research in the Mathematical Sciences},
    VOLUME = {9},
      YEAR = {2022},
    NUMBER = {2},
     PAGES = {Paper No. 22, 26},
      ISSN = {2522-0144,2197-9847},
   MRCLASS = {13A15 (05E40 13D02 13F20)},
  MRNUMBER = {4404868},
MRREVIEWER = {Houyi\ Yu},
       DOI = {10.1007/s40687-022-00316-4},
       URL = {https://doi.org/10.1007/s40687-022-00316-4},
}

@article {FHvT11,
    AUTHOR = {Francisco, Christopher A. and H\`a, Huy T\`ai and Van Tuyl,
              Adam},
     TITLE = {Colorings of hypergraphs, perfect graphs, and associated
              primes of powers of monomial ideals},
   JOURNAL = {J. Algebra},
  FJOURNAL = {Journal of Algebra},
    VOLUME = {331},
      YEAR = {2011},
     PAGES = {224--242},
      ISSN = {0021-8693,1090-266X},
   MRCLASS = {13F55 (05C17 05C65 13F20)},
  MRNUMBER = {2774655},
MRREVIEWER = {Amir\ Mafi},
       DOI = {10.1016/j.jalgebra.2010.10.025},
       URL = {https://doi.org/10.1016/j.jalgebra.2010.10.025},
}

@article {Scheja,
    AUTHOR = {Scheja, G\"unter},
     TITLE = {\"Uber die {B}ettizahlen lokaler {R}inge},
   JOURNAL = {Math. Ann.},
  FJOURNAL = {Mathematische Annalen},
    VOLUME = {155},
      YEAR = {1964},
     PAGES = {155--172},
      ISSN = {0025-5831,1432-1807},
   MRCLASS = {13.95 (13.90)},
  MRNUMBER = {162819},
MRREVIEWER = {M.\ Nagata},
       DOI = {10.1007/BF01344078},
       URL = {https://doi.org/10.1007/BF01344078},
}

@article {Ra13,
    AUTHOR = {Rauh, Johannes},
     TITLE = {Generalized binomial edge ideals},
   JOURNAL = {Adv. in Appl. Math.},
  FJOURNAL = {Advances in Applied Mathematics},
    VOLUME = {50},
      YEAR = {2013},
    NUMBER = {3},
     PAGES = {409--414},
      ISSN = {0196-8858,1090-2074},
   MRCLASS = {13F20 (05C25 13P10)},
  MRNUMBER = {3011436},
MRREVIEWER = {Siamak\ Yassemi},
       DOI = {10.1016/j.aam.2012.08.009},
       URL = {https://doi.org/10.1016/j.aam.2012.08.009},
}

@article {Omkar,
    AUTHOR = {Javadekar, Omkar},
     TITLE = {On {G}olod subdeterminantal ideals},
   JOURNAL = {arXiv:2601.18153 [math.AC]}
}
\end{document}